\documentclass[10pt,twoside]{article}

\usepackage{amsmath,amssymb,amsthm,amscd}
\usepackage[a4paper,text={140mm,240mm},centering,headsep=5mm,footskip=10mm]{geometry}
\usepackage[matrix,arrow]{xy}
\usepackage{color}
\usepackage{hyperref}
\usepackage[multiple]{footmisc}

\definecolor{green}{rgb}{0,0.7,0}
\definecolor{darkred}{rgb}{0.7,0,0}

\def\YEAR{\year}\newcount\VOL\VOL=\YEAR\advance\VOL by-1995
\def\firstpage{1}\def\lastpage{1000}
\def\received{}\def\revised{}
\def\communicated{}

\makeatletter
\def\magnification{\afterassignment\m@g\count@}
\def\m@g{\mag=\count@\hsize6.5truein\vsize8.9truein\dimen\footins8truein}
\makeatother

\font\eightrm=cmr8
\font\caps=cmcsc10                    
\font\Caps=cmcsc10 scaled \magstep1   

\makeatletter
\@specialpagefalse\headheight=8.5pt
\def\DocMath{}
\renewcommand{\@evenhead}{%
    \ifnum\thepage>\lastpage\rlap{\thepage}\hfill%
    \else\rlap{\thepage}\slshape\leftmark\hfill{\caps\SAuthor}\hfill\fi}%
\renewcommand{\@oddhead}{%
    \ifnum\thepage=\firstpage{\DocMath\hfill\llap{\thepage}}%
    \else{\slshape\rightmark}\hfill{\caps\STitle}\hfill\llap{\thepage}\fi}%
\makeatother

\def\TSkip{\bigskip}
\newbox\TheTitle{\obeylines\gdef\GetTitle #1
\ShortTitle  #2
\SubTitle    #3
\Author      #4
\ShortAuthor #5
\EndTitle
{\setbox\TheTitle=\vbox{\baselineskip=20pt\let\par=\cr\obeylines%
\halign{\centerline{\Caps##}\cr\noalign{\medskip}\cr#1\cr}}%
	\copy\TheTitle\TSkip\TSkip%
\def\next{#2}\ifx\next\empty\gdef\STitle{#1}\else\gdef\STitle{#2}\fi%
\def\next{#3}\ifx\next\empty%
    \else\setbox\TheTitle=\vbox{\baselineskip=20pt\let\par=\cr\obeylines%
    \halign{\centerline{\caps##} #3\cr}}\copy\TheTitle\TSkip\TSkip\fi%
\centerline{\caps #4}\TSkip\TSkip%
\def\next{#5}\ifx\next\empty\gdef\SAuthor{#4}\else\gdef\SAuthor{#5}\fi%
\ifx\received\empty\relax
    \else\centerline{\eightrm Received: \received}\fi%
\ifx\revised\empty\TSkip%
    \else\centerline{\eightrm Revised: \revised}\TSkip\fi%
\ifx\communicated\empty\relax
    \else\centerline{\eightrm Communicated by \communicated}\fi\TSkip\TSkip%
\catcode'015=5}}\def\Title{\obeylines\GetTitle}
\def\Abstract{\begingroup\narrower
    \parskip=\medskipamount\parindent=0pt{\caps Abstract. }}
\def\EndAbstract{\par\endgroup\TSkip}

\long\def\MSC#1\EndMSC{\def\arg{#1}\ifx\arg\empty\relax\else
     {\par\narrower\noindent%
     2000 Mathematics Subject Classification: #1\par}\fi}

\long\def\KEY#1\EndKEY{\def\arg{#1}\ifx\arg\empty\relax\else
	{\par\narrower\noindent Keywords and Phrases: #1\par}\fi\TSkip}

\newbox\TheAdd\def\Addresses{\vfill\copy\TheAdd\vfill
    \ifodd\number\lastpage\vfill\eject\phantom{.}\vfill\eject\fi}
{\obeylines\gdef\GetAddress #1
\Address #2 
\Address #3
\Address #4
\EndAddress
{\def\xs{4.3truecm}\parindent=0pt
\setbox0=\vtop{{\obeylines\hsize=\xs#1\par}}\def\next{#2}
\ifx\next\empty 
     \setbox\TheAdd=\hbox to\hsize{\hfill\copy0\hfill}
\else\setbox1=\vtop{{\obeylines\hsize=\xs#2\par}}\def\next{#3}
\ifx\next\empty 
     \setbox\TheAdd=\hbox to\hsize{\hfill\copy0\hfill\copy1\hfill}
\else\setbox2=\vtop{{\obeylines\hsize=\xs#3\par}}\def\next{#4}
\ifx\next\empty\ 
     \setbox\TheAdd=\vtop{\hbox to\hsize{\hfill\copy0\hfill\copy1\hfill}
                \vskip20pt\hbox to\hsize{\hfill\copy2\hfill}}
\else\setbox3=\vtop{{\obeylines\hsize=\xs#4\par}}
     \setbox\TheAdd=\vtop{\hbox to\hsize{\hfill\copy0\hfill\copy1\hfill}
	        \vskip20pt\hbox to\hsize{\hfill\copy2\hfill\copy3\hfill}}
\fi\fi\fi\catcode'015=5}}\gdef\Address{\obeylines\GetAddress}

\numberwithin{equation}{section}

\theoremstyle{plain}
\newtheorem{theo}{Theorem}[section]
\newtheorem*{theoNo}{Theorem}
\newtheorem{lem}[theo]{Lemma}

\newtheorem{prop}[theo]{Proposition}

\newtheorem{cor}[theo]{Corollary}
\newtheorem{claim}[theo]{Claim}

\theoremstyle{definition}
\newtheorem{defn}[theo]{Definition}

\theoremstyle{remark}
\newtheorem{rem}[theo]{Remark}

\newtheorem{exam}[theo]{Example}

\def\kb{\overline{k}}

\def\Sm{\mathcal Sm}

\def\Sch{\mathcal Sch}

\def\SPCk{{\mathcal SmProjCor} / k}
\def\SPCkZp{({\mathcal SmProjCor} / k)[\tfrac{1}{p}]}
\def\SP{{\ensuremath{\mathcal SmProj}}}
\def\SCk{{\mathcal SmCor} / k}
\def\Ck{{\mathcal Cor} / k}
\def\Comp{{\mathcal Comp}}
\def\Shv{{\mathcal Shv}}

\def\Dknz{D_{\et}(k,\nz)}

\def\DMeffk{DM^{\eff}(k)}

\def\DMeffket{DM^{\eff}_{\et}(k)}
\def\DMeffkR{DM^{\eff}(k,R)}
\def\DMeffkRet{DM^{\eff}_{\et}(k,R)}
\def\DMeffknz{DM^{\eff}(k,\nz)}
\def\DMeffknzet{DM^{\eff}_{\et}(k,\nz)}

\def\DMeffgmkR{DM^{\eff}_{gm}(k,R)}
\def\DMeffgmkZp{DM^{\eff}_{gm}(k,\zpi)}

\def\CHM{\textrm{Chow}^{\eff}(k, R)}
\def\CHMp{\textrm{Chow}^{\eff}(k, \zpi)}

\def\MHR#1#2{H^M_{#2}(#1,R)}

\def\MHRet#1#2{H^{M,\et}_{#2}(#1,R)}

\def\MHnz#1#2{H^M_{#2}(#1,\nz)}

\def\MHnzet#1#2{H^{M,\et}_{#2}(#1,\nz)}

\def\WHnz#1#2{H^W_{#2}(#1,\nz)}
\def\WcHnz#1#2{H^{W}_{#2,c}(#1,\nz)}

\def\wH#1#2{H^W_0(#1,#2)}
\def\wHH#1#2#3{H^W_{#1}(#2,#3)}

\def\Hnztet#1#2#3{H^{\et}_{#2}(#1,\nz(#3))}
\def\Hcnztet#1#2#3{H^{\et}_{#2,c}(#1,\nz(#3))}

\def\Xb{\overline{X}}

\newcommand{\rHom}{\operatorname{Rhom}}
\newcommand{\uHom}{\operatorname{\underline{hom}}}
\newcommand{\Tot}{\operatorname{Tot}}

\newcommand{\eff}{{\operatorname{eff}}}

\newcommand{\red}{{\operatorname{red}}}
\newcommand{\by}[1]{\overset{#1}{\longrightarrow}}
\newcommand{\iso}{\by{\sim}}

\renewcommand{\epsilon}{\varepsilon}
\renewcommand{\phi}{\varphi}

\newcommand{\Spec}{\mathrm{Spec}}

\newcommand{\Hom}{\text{\rm Hom}}

\newcommand{\Q}{{\mathbb Q}}
\newcommand{\Z}{{\mathbb Z}}

\newcommand{\im}{\mathrm{Im}}
\newcommand{\CH}{\mathrm{CH}}

\renewcommand{\CH}{\mathrm{CH}}
\newcommand{\et}{\text{\rm \'et}}
\newcommand{\qfh}{\text{qfh}}
\newcommand{\equi}{\text{equi}}

\newcommand{\ab}{\text{\rm ab}}

\def\CH{{\mathrm{CH}}}      
\def\Ker{{\mathrm{Ker}}}    
\def\tor{{\mathrm{tor}}}  
\def\et{{\text{\'{e}t}}}    
\def\cH{{\mathcal{H}}}        
\def\Hom{{\mathrm{Hom}}}    
\def\id{{\mathrm{id}}}     

\def\ker{{\mathrm{Ker}}}    

\def\ch{{\mathrm{ch}}}    
\def\Tr{{\mathrm{Tr}}}
\def\PT{{\textrm{Pre-Tr}}}

\def\cF{{\mathcal F}}
\def\cI{{\mathcal I}}

\def\cX{{\mathcal X}}

\def\cV{{\mathcal V}}

\def\cA{{\mathcal A}}

\def\cZ{{\mathcal Z}}
\def\cS{{\mathcal S}}
\def\cT{{\mathcal T}}
\def\cH{{\mathcal H}}

\def\AB{{\mathcal AB}}

\def\Lam{\Lambda}

\def\bZ{{\mathbb Z}}
\def\ZZ{{\mathbb Z}}
\def\bQ{{\mathbb Q}}

\def\bP{{\mathbb P}}

\def\bA{{\mathbb A}}

\def\nz{\bZ/n\bZ}

\def\zpi{\bZ[\tfrac{1}{p}]}
\def\zll{\bZ_{(l)}}

\def\rmapo#1{\overset{#1}{\longrightarrow}}

\def\isom{\overset{\cong}{\longrightarrow}}

\def\qaq{\quad\text{and}\quad}
\def\qfor{\quad\text{for }}
\def\qwith{\quad\text{with }}

\def\ol#1{\overline{#1}}

\def\Xb{\ol X}

\def\WHL#1#2{H^W_{#2}(#1,\Lam)}
\def\HCTR#1#2{\widetilde{H}^M_{#2}(#1,R)}

\newif\ifShowNewSectionTwoPointTwoChanges
\ShowNewSectionTwoPointTwoChangesfalse 

\begin{document}
\Title Weight homology of motives
\ShortTitle 
\SubTitle   
\Author Shane Kelly and Shuji Saito
\ShortAuthor 
\EndTitle
\Abstract 
In the first half of this article we define a new weight homology functor on Voevodsky's category of effective motives, and investigate some of its properties. In special cases we recover Gillet-Soul{\'e}'s weight homology, and Geisser's Kato-Suslin homology. In the second half, we consider the notions of ``co-{\'e}tale'' and ``reduced'' motives, and use the notions to a prove a theorem comparing motivic homology to {\'e}tale motivic homology.

Due to \cite{Ke} we do not have to restrict to smooth schemes.
\EndAbstract
\MSC 
14F43 
19E15, 
14F20, 
32S35. 
\EndMSC
\KEY 
Weight homology, Motivic cohomology, {\'E}tale cohomology, Finite fields.
\EndKEY
\Address 
Shane Kelly, 
Interactive Research Center of Science, 
Graduate School of Science and Engineering, 
Tokyo Institute of Technology 
2-12-1 Ookayama, Meguro, 
Tokyo 152-8551 JAPAN
shanekelly64@gmail.com
\Address
Shuji Saito,
Interactive Research Center of Science, 
Graduate School of Science and Engineering, 
Tokyo Institute of Technology 
2-12-1 Ookayama, Meguro, 
Tokyo 152-8551 JAPAN
sshuji@msb.biglobe.ne.jp
\Address
\Address
\EndAddress

\section{Introduction}

\subsection*{Weight homology}

In their highly cited work \cite[\S 2]{GS}, Gillet and Soul{\'e} associate to every variety $X$ over a perfect field%
\footnote{In their article they state their result for fields of characteristic zero since they use resolution of singularities, however if one uses $\zpi$-coefficients (resp. $\bQ$-coefficients) one can replace this with the appropriate theorem of Gabber (resp. de Jong). Alternatively, one can work under the assumption that the field admits resolution of singularities.} %
$k$, a complex $W(X)$ of (effective) Chow motives, and show it is well-defined up to chain homotopy. Using this, in \cite[\S 3]{GS} they define the \emph{weight homology} (with compact support) of $X$ with ($\Gamma$-coefficients) as
\[H^{GS}_n(X, \Gamma) = H_n  (\Gamma(W(X))  ) \]
 where $\Gamma: Chow^\eff \to \cA$ is an additive covariant functor from effective Chow motives $Chow^\eff$ to an abelian category $\cA$. A distinguished property of Gillet and Soul{\'e}'s weight homology is :
For $X$ projective smooth over $k$, we have
\[
H^{GS}_n(X, \Gamma)  =\left\{ \begin{array}{lr}
\Gamma(X) & \text{ for } n=0, \\
0 & \text{ for } n\ne 0. \end{array} \right.  
\]

On the other hand, for a variety $X$ over a finite field $k$, Geisser defines in \cite[\S 8]{Ge2} a homology theory he calls the Kato-Suslin homology with coefficients in an abelian group $A$ as
\[ H^{KS}_{n}(X, A) = H_n((\underline{C}_\ast(X)(\overline{k})\otimes A)_G) \]
where $\underline{C}_\ast(X)(\overline{k}) = \hom_{Cor}(\Delta^\ast_{\overline{k}}, \overline{k} {\times_k} X)$ is the Suslin complex of $\overline{X}$ and $G$
is the Weil group of $k$ (i.e., the subgroup of $Gal(\overline{k} / k)$ generated by the Frobenius). 
In \cite{Ge2} Geisser shows the following (see Prop.~\ref{Ge81}): 
Assume either that $A$ is torsion prime to $\ch(k)$, or that 
Parshin's conjecture $P_0$ stated in Prop.~\ref{Ge81}(\ref{Ge81:Q}) holds.
Then, for $X$ smooth over $k$, we have
\[
H^{KS}_{n}(X, A)  =\left\{ \begin{array}{lr}
A^{c(X)} & \text{ for } n=0 \\
0 & \text{ for } n\ne 0 \end{array} \right.  
\]
where $c(X)$ is the set of connected components of $X$.

\medbreak

The first main purpose of this article is to introduce a general homology theory for motives from 
which the above two can be recovered as special cases.

\ifShowNewSectionTwoPointTwoChanges
{ \color{blue}
\fi
Recall that in \cite{Bo2} Bondarko considers a dg-enhancement of the category of motives of smooth projective varieties (cf. Theorem~\ref{BondgEnhancement}, \cite[\S 2.4, \S 3.1]{Bo2}). We observe that a consequence of this is an induced equivalence of categories (see \S\ref{Notationconventions} for the notation):
\ifShowNewSectionTwoPointTwoChanges
}  
\fi

\begin{theoNo}[{Thm.~\ref{equivalenceHomologicalFunctors}}] \label{prop:homFunctorEquivalence}
Let $k$ be a perfect field of exponential characteristic $p$. The functor
\begin{equation} \label{equa:equivalences}
\left \{ \begin{array}{cc} 
\textrm{homological functors } \\ 
\textrm{on } \DMeffgmkZp
\end{array} \right  \}
\to
\left \{ \begin{array}{cc} 
\textrm{additive functors } \\ 
\textrm{on } \CHMp
\end{array} \right  \}
\end{equation}
induced by the inclusion $\CHMp \to \DMeffgmkZp$ induces an equivalence when restricted to the full subcategory of homological%
\footnote{By homological functor we mean a covariant functor with values in an abelian category which preserves sums and sends distinguished triangles to long exact sequences (see Def.~\ref{def(N)}).} %
 functors $\cH$ satisfying:
\begin{enumerate}
 \item[(SPA)] $\cH(M(X)[n]) = 0$ when $X$ is smooth and projective and $n \neq 0$.
\end{enumerate}
\end{theoNo}

\ifShowNewSectionTwoPointTwoChanges
{ \color{blue}
\fi
Thanks to the dg-enhancement that Bondarko considers, Theorem~\ref{equivalenceHomologicalFunctors} is a corollary of a more general version for negative dg-categories (Lemma~\ref{lemm:dgdecompose}). Alternatively, using Bondarko's weight structure on $\DMeffgmkZp$, one could also obtain Theorem~\ref{equivalenceHomologicalFunctors} as a corollary of \cite[Cor. 2.3.4]{Bo3}, a more general version for triangulated categories equipped with a weight structure.
\ifShowNewSectionTwoPointTwoChanges
} 
\fi

For any additive covariant functor $\Gamma: \CHMp \to \cA$ we write 
\begin{equation}\label{def.WH.intro}
\wH - \Gamma: \DMeffgmkZp \to \cA
\end{equation}
for its associated homological functor (Def.~\ref{def.weighthomology}) and write $\wHH n - \Gamma = \wH {-[n]} \Gamma$ for $n \in \bZ$. For an object $A$ of the category $\AB$ of abelian groups, we put $\wHH n - A =\wHH n - {\Gamma_A}$ where \mbox{$\Gamma_A:\CHMp \to \AB$} is the additive functor such that $\Gamma_A(X)=A^{c(X)}$ for $X$ smooth projective, where $c(X)$ is the set of connected components of $X$.

\ifShowNewSectionTwoPointTwoChanges
{ \color{blue}
\fi
Bondarko has observed that the composition $(t \circ M^c)(-)$ (where $t$ is his weight complex functor -- see Equation~\ref{equa:weightComplex}) is isomorphic to Gillet-Soul{\'e} weight complex (\cite[Prop. 6.6.2]{Bo2}). Using similar methods, we can show that Gillet-Soul{\'e}'s weight homology can be recovered from \eqref{def.WH.intro}. Furthermore, we also show that Geisser's Kato-Suslin homology can be recovered from \eqref{def.WH.intro}.
\ifShowNewSectionTwoPointTwoChanges
} 
\fi




\begin{theoNo}[{Thm. \ref{mainthm2}, Thm. \ref{mainthm1}}]
Let $k$ be a perfect field and $\Sch / k$ the category of finite type separated $k$-schemes.
\begin{itemize}
\item[(1)]
For an additive covariant functor $\Gamma: \CHMp \to \cA$, there are canonical isomorphisms for $X\in \Sch / k$
\[ H^{GS}_n(X, \Gamma) \cong H^W_n(M^c(X), \Gamma) \]
which are covariantly functorial for proper morphisms in $\Sch / k$.
\item[(2)]
Assume $k$ is a finite field. If $A$ is a torsion abelian group of torsion prime to the characteristic,
there are canonical isomorphisms for $X\in \Sch / k$
\begin{equation} \label{KSW}
H^{KS}_n(X, A) \cong H^W_n(M(X), A)
\end{equation}
which are covariantly functorial for morphisms in $\Sch / k$.
Moreover, for a general $\bZ[1/p]$-module $A$, the existence of isomorphisms such as \eqref{KSW} follows from Parshin's conjecture $P_0$.
\end{itemize}
\end{theoNo}

See \S\ref{WcHSch} and \S\ref{WHSch} for a number of properties that our weight homology (and consequently, Gillet-Soul{\'e} weight homology and Geisser's Kato-Suslin homology) satisfies. In particular, we notice that $H^W_n(M^c(-), \Gamma)$ is contravariant for quasi-finite flat morphisms in $\Sch / k$, and for arbitrary morphisms in $\Sm / k$ between schemes of the same dimension (our properties (Wc3) and (Wc4)). 
The corresponding contravariance of Gillet-Soul{\'e}'s theory $H^{GS}_n(-, \Gamma)$ seems to be new, and not at all obvious from its construction.

For functors $\Gamma$ on $\CHMp$ such as $\Gamma_A$ which are $\bP^1$-invariant, the associated weight homology functor satisfies a number of additional interesting properties: it vanishes on all twisted motives, it is birationally invariant on smooth (not necessarily projective) schemes, and the weight homology of (not necessarily projective) smooth schemes is concentrated in degree zero (cf. Lem.~\ref{HomFunctorProperties} and Thm.~\ref{thm.N2}).

\subsection*{Cycle maps}

For a ring $R$, we write $\DMeffkR$ for the tensor triangulated category of (unbounded) effective motives with $R$-coefficients constructed by Voevodsky \cite{V1}. The canonical functor
\[ \alpha^* : \DMeffkR \to \DMeffkRet \] 
connecting Voevodsky's category of motives and its {\'e}tale version admits a right adjoint $\alpha_*$. The motivic homology and {\'e}tale motivic homology with $R$-coefficients of an object $C$ in $\DMeffkR$ are defined as
\begin{align*}
\MHR C i &= \hom_{\DMeffkR}(R[i], C), \quad \textrm{ and } \\ 
\MHRet C i &= \hom_{\DMeffkR}(R[i], \alpha_*\alpha^*C)
\end{align*}
where $R[i] = M(k)[i]$ is a shift of the unit for the tensor structures. 
The second half of the article %
concerns the following maps functorial in $C\in\DMeffkR$: 
\[
\alpha^*: \MHR C i \to \MHRet C i 
\]
induced by the unit $\id \to \alpha_*\alpha^*$ of the adjunction $(\alpha^*, \alpha_*)$.
\medbreak

Recall that motivic homology is related to fundamental invariants of schemes. 
For example, if $p$ is invertible in $R$, or if we assume strong resolution of singularities,
for $X \in \Sch / k$, $r, i \in \Z$, 
there are canonical isomorphisms (\cite[\S 2.2]{V1} and \mbox{\cite[Chap. 5]{Ke}})
\begin{align}
\MHR {M(X)} i &\cong H_i^S(X, R) \label{MHS} \\
\MHR {M^c(X)(r)} i &\cong CH_{-r}(X, i + 2r; R) \label{MHhigherChow}
\end{align}
where $H_i^S(X, R)$ is the Suslin homology with $R$-coefficients (\cite{SV}), and $CH_\ast(X, \bullet; R)$ are the higher Chow groups with $R$-coefficients (\cite{B}).

On the ther hand, for $R=\nz$ with $n$ coprime to $p$, the \'etale motivic homology
$\MHnzet C i $ for $C=M(X)$ (resp. $M^c(X)$) are identified with the dual of \'etale cohomology
(resp. \'etale cohomology with compact support) (see Prop.~\ref{prop4.etreal} and Lem.~\ref{lem.ethom}). 

\medbreak

The main theorem is:

\begin{theoNo}[{Thm.~\ref{mainthm.etreal}}] \ 
Let $p$ be the exponential characteristic of $k$, let $n$ be an integer coprime to $p$, and let $R$ be $\bZ / n \bZ$ or $\bZ[1 / p]$.
\begin{itemize}
\item[(1)]
If $k$ is algebraically closed, $\alpha^*$ is an isomorphism.
\item[(2)]
If $k$ is finite and $C\in \DMeffgmkR$, there is a canonical long exact sequence
 \[
\cdots \to \MHR C i \rmapo{\alpha^*} \MHRet C i  \to \WHL C {i+1} \to \MHR C {i-1} \to \cdots
\]
which is functorial in $C$ where
\begin{equation*}
\Lam=
\left\{ \begin{array}{lr}
\bQ/\bZ[1/p] & \text{ if } R=\bZ[1/p],\\
\nz & \text{ if } R=\nz. \end{array} \right.  
\end{equation*}
\end{itemize}
\end{theoNo}

In view of the remarks above the theorem, the first part of the theorem implies that if $k$ is algebraically closed, then there are canonical isomorphisms for $X\in \Sch/k$:
\[
\begin{aligned}
 & H^S_i(X,\nz) \cong H^{i}(X_{\et},\nz)^*,  \textrm{ and }\\
 & \CH_0(X,i;\nz) \cong H^{i}_c(X_{\et},\nz)^*.
\end{aligned}
\] 
One can identify these isomorphisms with those in \cite{SV} and \cite{Su}. (This identification is essentially a consequence of the Suslin homology/cohomology pairing and the {\'e}tale homology/cohomology pairing being compatible. See \cite[Prop. 15 and 17]{Ke2} for the details).
Our proof of the above theorem is simpler than those in \cite{SV} and \cite{Su}, but depends on the Beilinson-Lichtenbaum conjecture due to Suslin-Voevodsky \cite{SVBlochKato} and Geisser-Levine \cite[Thm.1.5]{GL} which relies on the Bloch-Kato conjecture proved by Rost-Voevodsky (see the proof of Lem.~\ref{etcyclhigherChow}).

For a finite field $k$ and $R=\Lam=\nz$ with $n$ coprime to $p$, the map $\alpha^*$ gives rise to maps
\[
H^S_i(X,\nz) \to H^{i+1}(X_{\et},\nz)^*,\quad \CH_0(X,i;\nz) \to H^{i+1}_c(X_{\et},\nz)^*,
\] 
and the second part of the above theorem relates the kernel and cokernel of these maps with weight homology and
weight homology with compact support of $X$ introduced in \S\ref{wthomologysch}.
In particular, if $X$ is smooth over $k$, we get canonical isomorphisms
\begin{equation*}\label{HSetcohsmooth.intro}
H^S_i(X,\nz) \cong H^{i+1}(X_{\et},\nz)^*\qfor i\in \bZ.
\end{equation*}
See Cor.~\ref{cor.etreal} for the details and Rem.~\ref{rem.Geisser} for relation of these results with 
\cite{Ge2}.

\medbreak

In order to show the second part of the above theorem,
we introduce new homology functors which seem to be of independent interest. 
The idea is to introduce an endo-functor $\Theta:\DMeffkR\to \DMeffkR$ equipped with a natural transformation
$\Theta\to id$, which gives rise to  
distinguished triangles (Equation~\ref{Thetadt}) 
\[ \Theta C \to C \to \alpha_*\alpha^*C \to \Theta C[1] \]
for $C\in \DMeffkRet$, allowing one to define the ``co-{\'e}tale'' part $\Theta C$ of a motive $C$. 
Then the second part of the above theorem is equivalent to the assertion that 
if $k$ is finite, there are canonical isomorphisms for $C\in \DMeffgmkR$: 
\begin{equation}\label{eqintro.thm.wHTheta} 
\MHR {\Theta(C))} i \isom \WHL C {i+1}\qfor i\in \bZ,
\end{equation} 
which are natural in $C$ (see Theorem \ref{thm.wHTheta}). 
To show this, we use $d_{\le 0}\DMeffkR$, the localizing subcategory of $\DMeffkR$ generated by motives of smooth varieties of dimension zero (cf. \cite[\S 3.4]{V1}).
By \cite{abv} there is a projection functor $L\pi_0: \DMeffkR\to d_{\le 0}\DMeffkR$, the ``field of constants'' part 
(cf. Equation~\eqref{pi0M(X)}).
Similarly as above, we get again distinguished triangles (Equation~\ref{redTri})
\[ C_{red} \to C \to L\pi_0(C) \to C_{red}[1] \]
for $C\in \DMeffkRet$, allowing one to define the ``co-dimension-zero'' or ``reduced'' part $C_{red}$ of a motive $C$. 
Surprisingly, if $k$ is a finite field, $\MHR {\Theta(C_{red})} i=0$ for all $i\in \ZZ$,
or equivalently the canonical morphism
\[ \Theta (C) \to \Theta(L\pi_0(C)) \]
induces isomorphisms $\MHR {\Theta (C) } i\cong \MHR {\Theta (L\pi_0(C)) } i$ for all $i\in \ZZ$ (Claim \ref{claim1.etreal}(2), Equation~\eqref{piNotNec}).
Heuristically, this can be interpreted as saying that over a finite field, all the difference between the motivic homology and the {\'e}tale motivic homology is contained in the field of constants.
Finally the desired isomorphisms \eqref{eqintro.thm.wHTheta} are obtained by showing 
canonical isomorphisms for $C\in \DMeffgmkR$: 
\begin{equation}\label{eqintro1.thm.wHTheta} 
\MHR {\Theta(L\pi_0(C))} i \isom \WHL C {i+1}\qfor i\in \bZ,
\end{equation} 
which are natural in $C$ (see \eqref{piNotwHomology}).

\subsection*{Outline}

In Section~\ref{weighthomology} we present our notation and conventions, and then the definition of our weight homology of motives functor (Def.~\ref{def.weighthomology}). We also discuss the relationship between some ``birational'' and ``weight'' properties homological functors might satisfy in Lemma~\ref{HomFunctorProperties}. Our weight homology functor satisfies all of them (Theorem~\ref{thm.N2}).

In Section~\ref{wthomologysch} we specialise to the case of weight homology of motives of schemes and motives of schemes with compact support. It is in this section that we show that our theory recovers the Gillet-Soul{\'e} weight homology over a perfect field (Thm.~\ref{mainthm2}), and Geisser's Kato-Suslin homology over a finite field if the coefficient group is finite torsion prime to the characteristic of the base field (Thm.~\ref{mainthm1}).

In Section~\ref{sec:homology} we present our main comparison theorem between motivic and {\'e}tale motivic homology (Theorem~\ref{mainthm.etreal}) and give some consequences. The theorem is proved in the subsequent sections. Section~\ref{etcyclemap} is the comparison in the case of motives of smooth schemes with compact support (Theorem~\ref{etcycl-propersmooth}).

In Section~\ref{effmotive} we present the notions of ``co-{\'e}tale'' and ``reduced'' motives and use them to prove the general case of Theorem~\ref{mainthm.etreal} (see Theorem~\ref{thm.wHTheta}).


In Appendix~\ref{etalerealisationnonsmooth} we identify that the {\'e}tale realisation of motives of singular schemes (and motives with compact support of singular schemes) are what one expects them to be.

\subsection*{Acknowledgements}

We are indebted to  B. Kahn for some ideas used in this paper.
In particular, the idea of using the reduced part of a motive to study the cycle map over a finite field is due to him. We thank T. Geisser for helping us to understand the relationship between our work and his. We also thank M. Bondarko for helpful comments.

\subsection*{Warning}

We almost always work under the hypothesis that either the base field satisfies resolution of singularities, or alternatively, that the coefficients ($A$, $\Lambda$, $R$, $\nz$, or $\Gamma$ depending on the statement) are $\zpi$-linear. We will  remind the reader from time to time.



\section{Weight homology of motives} \label{weighthomology}
\def\cB{{\mathcal B}}
\def\WcHG#1#2{H^W_{#2,c}(#1,\Gamma)}

\subsection{Notation and conventions}\label{Notationconventions}

Through-out, $k$ will be a perfect field of exponential characteristic $p$ and $n$ a positive integer prime to $p$. We will remind the reader occasionally. For a ring $R$, we write $\DMeffkR$ for the tensor triangulated category of (unbounded) effective motives with $R$-coefficients constructed by Voevodsky in \cite{V1}. We will use $\Sm / k$ (resp. $\Sch / k$) to denote the category of smooth (resp. separated finite type) $k$-schemes.
If $p$ is invertible in $R$, or if we assume strong resolution of singularities the category $\DMeffkR$ is equipped with functors
\begin{align} 
M: \Sch / k \to &\DMeffkR \ ;  &X \mapsto M(X) \label{nonsmoothmotives}  \\
M^c: \Sch^{\mathrm{prop}} / k \to &\DMeffkR \ ;  &X \mapsto M^c(X) \label{motiveCompact}
\end{align}
where $\Sch^{\mathrm{prop}} / k$ is the subcategory of $\Sch / k$ with all objects but only proper morphisms \cite[Chap. 5]{Ke} (or \cite{V1} for the version without conditions on the coefficients, but which assumes resolution of singularities).
\medbreak

We now focus on the subcategory $\DMeffgmkR$ of compact objects of $\DMeffkR$. 
Let $\SCk$ be Voevodsky's category of smooth correspondances and $K^b(\SCk)$ be the homotopy category of 
bounded complexes in $\SCk$. Then there is a natural functor (cf. \cite{V1})
\begin{equation}\label{KDM}
\pi: K^b(\SCk)\to \DMeffgmkR 
\end{equation}
such that $M(X)=\pi(X[0])$ for $X\in \Sm$. 
We also consider the category of (covariant) effective Chow motives $\CHM$. One description of $\CHM$ is as the smallest idempotent complete full subcategory of $\DMeffkR$ containing the objects $M(X)$ for all $X$ smooth and projective (\cite[5.5.11(4)]{Ke} or \cite[4.2.6]{V1}). From this description there is a tautological inclusion 
\begin{equation} \label{CHMDMinclusion}
\iota: \CHM \to \DMeffgmkR.
\end{equation}

\ifShowNewSectionTwoPointTwoChanges
{ \color{blue}
\fi
We also have cause to use the full subcategory of $\SCk$ whose objects are smooth projective $k$-varieties, for which we write $\SPCk$. From this we also have $\SPCkZp$, the category with the same objects as $\SPCk$ and 
\[ \hom_{\SPCkZp}(X, Y) = \hom_{\SPCk}(X, Y) \otimes \zpi. \]

\subsection{Differential graded motives and weight homology} \label{weighStructures}


One of the important tools in \cite{Bo2} is the following theorem. We refer the reader to Appendix~\ref{dgcategoriesSection} for some recollections about dg-categories, notation, and conventions.

\begin{theo}[Bondarko {\cite[\S 2.4, \S 3.1]{Bo2}}] \label{BondgEnhancement}
There exists a negative dg-category $J$ such that $J^0 = \SPCkZp$, and a commutative diagram
\[ \xymatrix{
H^0(J) \ar[r] \ar[d]_{\cong} & \Tr(J) \ar[d]^{\cong} \\
\CHMp \ar[r] &  \DMeffgmkZp
} \]
where the horizontal functors are the canonical ones, the right equivalence is a triangulated functor, and the left one is induced by the canonical functor $\SPCkZp \to \CHMp$.
\end{theo}



A consequence of Bondarko's Theorem 
\ref{BondgEnhancement} is Theorem~\ref{equivalenceHomologicalFunctors} below. We obtain this theorem using the dg-enhancement from Theorem \ref{BondgEnhancement} and ``b{\^e}te'' truncations of twisted complexes via Lemma~\ref{lemm:dgdecompose}. Alternatively, one could also obtain it using Bondarko's weight structure on $\DMeffgmkZp$ together with his general weight structure equipped triangulated category version of the theorem (\cite[Cor. 2.3.4]{Bo3}).%
\footnote{In fact, the main subject of \cite{Bo3} is the construction of new cohomology theories on $SH^{S^1}(k)$ via \cite[Cor. 2.3.4]{Bo3}.} %

\begin{defn}\label{def(N)}
Let $\cT$ be a triangulated category and $\cA$ be an abelian category.
\begin{itemize}
\item[(1)]
A \emph{homological functor} $\cH:\cT\to \cA$ is an additive functor (preserving all small sums that exist in $\cT$) such that for any distinguished triangle
$T_1\to T_2\to T_3 \rmapo + T_1[1]$,
the sequence
\[
\cdots \to \cH(T_1) \to \cH(T_2) \to \cH(T_3) \to \cH(T_1[1]) \to \cH(T_2[1]) \to \cdots
\]
is exact. 
\item[(2)]
A homological functor $\cH: \DMeffgmkR \to \cA$ is \emph{\SP-acyclic} if:
\begin{enumerate}
 \item[(SPA)] for every smooth projective $X$ we have $\cH(M(X)[i]) = 0$ for $i \neq 0$.\end{enumerate}
\end{itemize}
\end{defn}

\begin{theo} \label{equivalenceHomologicalFunctors}
Let $k$ be a perfect field of exponential characteristic $p$. Composition with the inclusion $\CHMp \stackrel{\iota}{\to} \DMeffgmkZp$ induces an equivalence of categories
\begin{equation} \label{equa:equivalences}
\iota^* : 
\left \{ \begin{array}{cc} 
\SP\textrm{-acyclic  } \\ 
\textrm{homological functors } \\ 
\DMeffgmkZp \to \cA \\
\end{array} \right  \}
\to
\left \{ \begin{array}{cc} 
\textrm{additive functors } \\ 
\CHMp \to \cA 
\end{array} \right  \}.
\end{equation}
Furthermore, if $V^\bullet$ is a bounded complex in $\SPCk$ and $M(V^\bullet)$ its image in $\DMeffgmkZp$ then there are canonical isomorphisms
\begin{equation} \label{wssIsos}
\cH(M(V^\bullet)[n]) \cong \frac{\ker(\cH(V^n) \stackrel{}{\to} \cH(V^{n + 1}))}{\im(\cH(V^{n-1}) \stackrel{}{\to} \cH(V^n))}
\end{equation}
where $\cH$ is an object of the left category. These isomorphisms are natural in $\cH$ and in $V^\bullet \in Comp^b(\SPCk)$.
\end{theo}

\begin{proof}
This is an immediate consequence of Theorem~\ref{BondgEnhancement} and Lemma~\ref{lemm:dgdecompose}.
\end{proof}

\begin{lem} \label{lemm:dgdecompose}
Suppose that $J$ is a negative differential graded category and that $\cH: \Tr(J) \to \cA$ is a $J$-acyclic homological functor. That is, $\cH$ is a homological functor such that
\begin{enumerate}
 \item[($J$-A)] for $P^0 \in J$ and $n \neq 0$ we have $\cH(P^0[n]) = 0$.
\end{enumerate}
Then for every twisted complex $P = \{(P^i), q_{ij}\}$ there are canonical isomorphisms
\[ \cH(P) \cong \frac{\ker(\cH(P^0) \stackrel{\cH(q_{01})}{\to} \cH(P^1))}{\im(\cH(P^{-1}) \stackrel{\cH(q_{-1,0})}{\to} \cH(P^0))} \]
natural in $P \in \Tr(J)$, and $\cH$. Consequently, composition with the canonical functor $H^0(J) \to \Tr(J)$ induces an equivalence from the category of $J$-acyclic homological functors $\cH: \Tr(J) \to \cA$ to the category of additive functors $H^0(J) \to \cA$.
\end{lem}

In the proof of Lemma~\ref{lemm:dgdecompose} we will freely use definitions and notation described in Section~\ref{dgcategoriesSection}. In particular, we will make use of the following canonical distinguished triangles, reproduced here for ease of reference.

 \begin{align} 
P_{\geq n}
\  \stackrel{\iota}{\to}\ 
P
\ \stackrel{\pi}{\to}\ &
P_{\leq n - 1}
\ \stackrel{\delta}{\to}\ 
P_{\geq n}[1] 
\tag{\ref{equa:decompTriTruncate}} \\
P_{\geq m}
\  \stackrel{\pi'}{\to} \ 
P^m[-m]
\ \stackrel{\delta'}{\to}\ &
P_{\geq m + 1}[1]
\ \stackrel{\iota'}{\to}\ 
P_{\geq m}[1] 
\tag{\ref{equa:geqTri}} \\
P_{\leq m}
\ \stackrel{\delta''}{\to}\ 
P^{m + 1}[-m]
\ \stackrel{\iota''}{\to}\ &
P_{\leq m + 1}[1]
\ \stackrel{\pi''}{\to}\ 
P_{\leq m}[1] 
\tag{\ref{equa:leqTri}}
\end{align}

\begin{proof}
We claim that in the diagram
 \begin{equation} \label{equa:asdf}
\xymatrix{
\cH(P^{ - 1}) \ar[r]^{\cH(\delta')} & \cH(P_{\geq 0 }) \ar[d]_{\cH(\pi')} \ar[r]^{\cH(\iota)} & \cH(P) \ar[r] & 0 \\
 & \ker \left (\cH P^0 \stackrel{\cH(\pi'\delta')}{\to} \cH P^{1} \right ), & 
}
\end{equation}
the row is exact and the morphism $\cH(\pi')$ is an isomorphism, and consequently, the three morphisms in the diagram induce a canonical isomorphism
\[ \frac{\ker(\cH P^0 \to \cH P^{1})}{\im(\cH P^{- 1})}  
\underset{\cH(\pi')}{\stackrel{\sim}{\leftarrow}} 
\frac{\cH(P_{\geq 0})}{\im( \cH P^{ - 1}) } 
\underset{\cH(\iota)}{\stackrel{\sim}{\to}}
\cH(P). \]
These isomorphisms are natural due to the naturality described in Equation~\eqref{equa:functoriotapi}.

To prove this claim, we begin by noticing that using the triangles from Equation~\eqref{equa:geqTri} and \eqref{equa:leqTri} and induction on the length of the twisted complex, one sees that
\begin{equation} \label{equa:acyVandg}
\cH(P_{\leq n}[k]) = 0 \textrm{ for } k > n, \qquad \textrm{ and } \qquad \cH(P_{\geq n}[k]) = 0 \textrm{ for } k < n. 
\end{equation}
By the vanishing \eqref{equa:acyVandg} and the triangles \eqref{equa:decompTriTruncate} and \eqref{equa:leqTri} there are two exact sequences
\begin{align}
&\cH(P_{\leq - 1}[ {-} 1]) \stackrel{\cH(\delta)}{\to} \cH(P_{\geq 0}) \stackrel{\cH(\iota)}{\to} \cH(P) \to 0  \nonumber \\   
\cH(P^{ - 1}) \stackrel{\cH(\iota'')}{\to} &\cH(P_{\leq  - 1}[ - 1]) \to 0. \nonumber
\end{align}
which shows that the row of \eqref{equa:asdf} is exact since $\delta' = \delta \iota''$. %
%
Then by the vanishing \eqref{equa:acyVandg}, and the triangle \eqref{equa:geqTri}, we get the two exact sequences
\begin{align}
0 \to \cH(P_{\geq 0}) \stackrel{\cH(\pi')}{\to} \cH P^0 \stackrel{\cH(\delta')}{\to} &\cH(P_{\geq 1}[1])  \nonumber \\
0 \to &\cH(P_{\geq 1}[1]) \stackrel{\cH(\pi')}{\to} \cH P^{1} \nonumber 
\end{align}
and hence $\cH(P_{\geq 0})$ is equal to the kernel of $\cH P^0 \stackrel{\cH(\delta')}{\to}\stackrel{\cH(\pi')}{\to} \cH P^{1}$. That is, the morphism $\cH(\pi')$ of \eqref{equa:asdf} is an isomorphism.
\end{proof}

\ifShowNewSectionTwoPointTwoChanges
} 
\fi

\begin{defn}\label{def.weighthomology}
For an additive functor $\Gamma: \CHMp \to \cA$, let 
\[\wH - \Gamma: \DMeffgmkZp \to \cA\]
be the \SP-acyclic (Def.~\ref{def(N)}) homological functor associated to $\Gamma$ by Theorem \ref{equivalenceHomologicalFunctors}.
We put
\[
\wHH i - \Gamma = \wH {-[i]} \Gamma \qfor i\in \bZ.
\]

\ifShowNewSectionTwoPointTwoChanges
{\color{blue} 
\fi

\begin{rem}
An inverse to Equation~\eqref{equa:equivalences} is given by $\Gamma \mapsto (H^0\Gamma) \circ t$ where \mbox{$H^i\Gamma: K^b(\CHMp) \to \cA$} is the $i$-th cohomology of $\Gamma$ applied term-wise to complexes in $K^b(\CHMp)$ and $t$ is Bondarko's weight complex functor (Equation~\eqref{equa:weightComplex}). So, with this notation, we have
\begin{equation} \label{eq.wthomc.alt}
\wHH i - \Gamma= (H^i\Gamma) \circ t.
\end{equation}
\end{rem}

\ifShowNewSectionTwoPointTwoChanges
} 
\fi

For an abelian group $A$, we put
\[\wHH i - A =\wHH i - {\Gamma_A},\]
where $\Gamma_A:\CHMp \to \AB$ is the additive functor such that $\Gamma_A(X)=A^{c(X)}$ for $X$ smooth projective,
where $c(X)$ is the set of connected components of $X$.
\end{defn}


\subsection{Additional properties of homological functors}

We will observe a wide class of homological functors on $\DMeffgmkR$ satisfies stronger conditions than $\SP$-acyclicity (Def.~\ref{def(N)}) (see Theorem \ref{thm.N2} below):

\begin{lem} \label{HomFunctorProperties}
Suppose either resolution of singularities holds, or $R$ is a $\zpi$-algebra. Let $\cH: \DMeffgmkR \to \cA$ be a homological functor. The following three properties are equivalent.%
\footnote{(PI) $= $ $\bP^1$-Invariance, (TL) $=$ Tate Local, (BI) $=$ Birational Invariance.}
\begin{itemize}
 \item[(PI)] The canonical morphism $\cH( M(\bP^1_X)[i] ) \to \cH( M(X)[i] )$ is an isomorphism for all smooth projective $X$, and all $i \in \bZ$.
 \item[(TL)] We have $\cH(M(1)[i]) = 0$ for all $M \in \DMeffgmkR$ and $i \in \bZ$.
 \item[(BI)] for any dense open immersion $U \to X$ of smooth schemes the induced morphism $\cH( M(U)[i] ) \to \cH( M(X)[i] )$ is an isomorphism for all $i \in \bZ$.
\end{itemize}
Moreover, a $\SP$-acyclic homological functor $\cH$ satisfies the above equivalent properties if and only if it is \emph{$\Sm$-acyclic}, i.e., if and only if:
\begin{itemize}
 \item[($\Sm$A)] For every smooth (not necessarily projective) $X$ we have $\cH(M(X)[i]) = 0$ for $i \neq 0$.
\end{itemize}
\end{lem}

\begin{rem}
The equivalence (BI) $\Leftrightarrow$ (TL) is essentially shown in \cite[Prop. 5.2(b)]{KaSu2} using the same techniques, but we include the proof for the convenience of the reader.
\end{rem}
 
\begin{proof}
(PI) $\Rightarrow$ (TL). Since motives of smooth projective varieties generate $\DMeffgmkR$ as a triangulated category it follows that \mbox{$\cH( M(\bP^1) \otimes M' ) \to \cH( M' )$} is an isomorphism for all objects $M' \in \DMeffgmkR$. Now the isomorphism \mbox{$M(\bP^1) = \bZ \oplus \bZ(1)[2]$} implies the claim.

(TL) $\Rightarrow$ (BI). This follows from the generalised Gysin distinguished triangle (for $U \to X$ a dense open immersion of smooth schemes of dimension $d$ with closed complement $Z$)
\[ M(U) \to M(X) \to M^c(Z)^*(d)[2d] \to M(U)[1] \]
(\cite[Prop 5.5.5, Thm 5.5.14(3)]{Ke}, cf. \cite[Prop. 4.1.5, Thm 4.3.7(3)]{V1}).

(BI) $\Rightarrow$ (PI). This follows from a very neat geometric argument of Colliot-Th{\'e}l{\`e}ne. See \cite[Appendix A]{KaSu}.

Now suppose that $\cH$ is $\SP$-acyclic.

(BI) $\Rightarrow$ ($\Sm$A). If resolution of singularities holds, this is a direct consequence of the existence of smooth compactifications. Otherwise: By Nagata's compactification theorem and Gabber's alterations theorem \cite{Il} for every smooth scheme $U$ and every $l \neq p$ there is a roof of morphisms $U \leftarrow U' \to X'$ where $X'$ is smooth and projective, $U' \to X'$ is a dense open immersion, and $U' \to U$ is finite flat surjective of degree prime to $l$. Since $\cH$ is assumed to be $\SP$-acyclic, the property (BI) implies that ($\Sm$A) is satisfied for the scheme $U'$. But $M(-)$ is functorial in $\SCk$ so by the trace formula, $M(f)M({^tf}) = \deg f \cdot \id_{M(U)}$, for the finite flat morphism $f$, the motive $M(U)$ is a direct summand of $M(U')$, at least $\zll$-linearly. But we were free to choose $l \neq p$, and so $M(U)$ is a direct summand of $M(U')$ as long as we use $\zpi$-linear coefficients (cf. \cite[Appendix A2]{Ke}). Hence, Property (TL) for $U'$ implies Property (TL) for $U$.

($\Sm$-acyclic) $\Rightarrow$ (TL). Let $X \in \Sm / k$. Since $M(\bP^1_X) \cong M(X) \oplus M(X)(1)[2]$, the property ($\Sm$A) implies that $\cH(M(X)(1)[i]) = 0$ unless $i = 2$. But the isomorphism \mbox{$M(\bA^1_X - X) \cong M(X) \oplus M(X)(1)[1]$} with ($\Sm$A) then implies that $\cH(M(X)(1)[i]) = 0$ unless $i = 1$. That is, $\cH(M(X)(1)[i]) = 0$ for all $i$. Since objects of the for $M(X)$ for $X \in \Sm / k $ generate the triangulated category $\DMeffgmkR$ it follows that $\cH(M(1)[i]) = 0$ for all $M \in \DMeffgmkR$ and $i \in \bZ$.
\end{proof}
\medskip

\begin{theo}\label{thm.N2}
Suppose either resolution of singularities holds, or $A$ is a $\zpi$-module. The homological functor $\wH - A$ from Definition \ref{def.weighthomology} 
satisfies all properties of Lemma \ref{HomFunctorProperties}. That is, $\wH - A$ is $\Sm$-acyclic.
\end{theo}

\begin{proof}
Obviously $\Gamma_A(X)=\Gamma_A(\bP^1_X)$ for $X$ smooth projective so that $\wH - A$ satisfies (PI). 
It also satisfies (N) by definition.
Hence Theorem \ref{thm.N2} follows from Lemma \ref{HomFunctorProperties}.
\end{proof}

\begin{rem}
Saying that the functor $\wH - A$ satisfies the three equivalent properties (PI), (TL), and (BI) is the same as saying that it factors through Kahn-Sujatha's triangulated category of birational motives (\cite[\S 5]{KaSu2}).
\end{rem}



\section{Weight homology of schemes}\label{wthomologysch}

As always, let $k$ be a perfect field of exponential characteristic $p$. 
In this section we fix an additive functor $\Gamma: \CHMp \to \cA$.

\subsection{Weight homology with compact support of schemes}\label{WcHSch}

In this subsection we compare our weight homology of a motive with compact supports to Gillet-Soul{\'e}'s weight homology, and we observe some consequences.

For $X\in \Sch / k$ and $i \in \Z$, we put (cf. Def.~\ref{def.weighthomology})
\begin{equation}\label{eq.wthomc}
\WcHG X i = \wH {M^c(X)[i]} \Gamma.
\end{equation}

Let $\Sch^{\mathrm{prop}} / k$ be the category of the same objects as $\Sch / k$ but only proper morphisms.
By \cite[\S2.2]{V1} (if resolution of singularities holds) or \cite{Ke}, this gives rise to covariant functors 
\[
\WcHG - i\;:\; \Sch^{\mathrm{prop}} / k  \to \cA\quad (i\in \Z)
\]
satisfying the following properties:
\begin{itemize}
\item[$({\bf Wc}1)$](Nilpotent invariance)
For $X \in \Sch / k$ with its reduced structure $X_{\red}\hookrightarrow X$, 
the map $\WcHG {X_{\red}} i\to \WcHG X i$ is an isomorphism for any $i\in \Z$.
\item[$({\bf Wc}2)$](Localization axiom)
For $X\in \Sch / k$ and a closed subscheme $Z\subset X$, one has a functorial long exact sequence
\[
\cdots\to \WcHG {Z} i \to \WcHG {X} i  \to \WcHG {X-Z} {i} \to \WcHG {Z} {i-1} \cdots
\]
\item[$({\bf Wc}3)$](Pull-back by quasi-finite flat morphisms)
For a quasi-finite flat morphism $f:Y\to X$ in $\Sch / k$, one has a canonical map
\[
f^*: \WcHG {X} i \to \WcHG {Y} i
\]
compatible with functoriality for proper morphisms in $\Sch / k$ and with localization sequences in $({\bf Wc}2)$ in an obvious sense.
\item[$({\bf Wc}4)$](Pull-back by morphisms between smooth schemes) This is a consequence of duality.
For a morphism \mbox{$f:Y\to X$} in $\Sm / k$ such that $\dim(Y)=\dim(X)$, one has a canonical map
\[
f^*: \WcHG {X} i \to \WcHG {Y} i
\]
compatible with functoriality for proper morphisms in $\Sch / k$ and with localization sequences in $({\bf Wc}2)$ in an obvious sense. 

\item[$({\bf Wc}5)$]
For $X\in \Sch / k$ projective smooth, we have
\[
\WcHG X i =\left\{ \begin{array}{lr}
\Gamma(X) & \text{ for } i=0 \\
0 & \text{ for } i\ne 0 \end{array} \right.  
\]

\item[$({\bf Wc}6)$] 
If the groups $\Gamma(X)$ are finitely generated for all $X$ smooth and projective (e.g., $\Gamma = \Gamma_{\bZ / n}$) then the groups $\WcHG X i$ are finitely generated for all $X \in \Sch / k$.
\end{itemize}

\bigskip

Recall the definition of the Gillet-Soul{\'e} weight homology: For $X \in Sch / k$, let $X \to \overline{X}$ be an open immersion into a projective variety and $Z$ the closed complement. Take $\overline{\cX}_\bullet \to \overline{X}$ a proper cdh hypercover (or proper $l$dh-hypercover \cite{Il}, \cite{Ke}) such that each $\overline{\cX}_i$ is smooth over $k$ and let $\cZ_\bullet \to Z$ be a refinement of the induced hypercover of $Z$ such that each $\cZ_i$ is smooth. We then obtain a complex concentrated in positive homological degrees in the additive envelope of the category of smooth projective varieties equipped with an augmentation
\begin{equation} \label{smoothcdhReplacement}
\cV_\bullet = Cone( \cZ_\bullet \to \overline{\cX}_\bullet) \to X.
\end{equation}
By definition, the Gillet-Soul{\'e} weight complex $W(X)$ is the image of $\cV_\bullet$ in the homotopy category of complexes of Chow motives $K^b(\CHMp)$, and for an additive functor $\Gamma: \CHMp \to \AB$ one defines
\begin{equation}\label{HGS.eq}
H^{GS}_i(X,\Gamma) = H_i(\Gamma(\cV_\bullet)).
\end{equation}


\begin{theo}\label{mainthm2}
There are isomorphisms
\begin{equation} \label{GSWHcisos}
H^{GS}_{i }(X,\Gamma) \cong \WcHG X {i}
\end{equation}
natural in $X \in \Sch^{\mathrm{prop}} / k$.
\end{theo}

\begin{rem}
A remarkable point is that for $H^{GS}_i(-,\Gamma)$, the properties $({\bf Wc}3)$ and $({\bf Wc}4)$ are not obvious from its construction.
\end{rem}

\ifShowNewSectionTwoPointTwoChanges
{ \color{blue}
\fi

\begin{rem}
Using the description of Equation~\eqref{eq.wthomc.alt}, Theorem~\ref{mainthm2} also follows from Bondarko's result that the Gillet-Soul{\'e} weight complex functor is isomorphic to $(t \circ M^c)(-)$ (\cite[\S 6.6]{Bo2}). They key point of his comparison are the isomorphisms \mbox{$M(\cV_\bullet) {\to} M^c(X)$} coming from cdh descent and the localisation distinguished triangles of $\DMeffgmkZp$, so this amounts to essentially the same proof.
\end{rem}

\ifShowNewSectionTwoPointTwoChanges
} 
\fi

\begin{proof}
The isomorphisms \eqref{GSWHcisos} come from Equation~\eqref{wssIsos} and the morphism \mbox{$M(\cV_\bullet) {\to} M^c(X)$} in $\DMeffgmkZp$ induced by Equation~\eqref{smoothcdhReplacement} which is an isomorphism by \cite[Thm 5.3.1]{Ke} (cf. \cite[Thm 4.1.2]{V1}).
\end{proof}

\medbreak

\begin{exam}\label{WHexam1}
Let $E\in \Sch / k$ be projective which is a simple normal crossing divisor on a smooth $P\in \Sch / k$.
Let $E_1,\dots, E_N$ be the irreducible components of $E$ and put for an integer $a\geq 0$
$$
 E^{[a]} =\underset{1\leq i_0< \cdots< i_a\leq N}{\coprod} E_{i_0,\dots,i_a}
\quad (E_{i_0,\dots,i_a} = E_{i_0}\cap\cdots\cap E_{i_a}).
$$
For $\Gamma:\CHM \to \cA$ as above, consider a homological complex in $\cA$
$$ 
\Gamma(E) : \cdots\to \Gamma(E^{[a]}) \rmapo{\partial_a} \Gamma(E^{[a-1]}) 
\rmapo{\partial_{a-1}} \cdots \rmapo{\partial_1} \Gamma(E^{[0]}), 
$$
where 
$\Gamma(E^{[a]})$ is in degree $a$ and $\partial_a$ is the alternating sum 
of the maps induced by inclusions $E_{i_0,\dots,i_a}\hookrightarrow E_{i_0,\dots,\widehat{i_\nu},\dotsi_a}$
for $\nu=0,\dots,a$. Then, using $({\bf Wc}2)$ and $({\bf Wc}5)$, one can show isomorphisms
\begin{equation*}\label{WHexam1}
    \WcHG E {a} \cong H_{a}(\Gamma(E))\qfor a\in \Z. 
\end{equation*}
\end{exam}

\subsection{Weight homology of schemes}\label{WHSch} 
\def\WHG#1#2{H^W_{#2}(#1,\Gamma)}

For $X\in \Sch / k$ and $i\in \Z$, we put (cf. Definition \ref{def.weighthomology})
\begin{equation}\label{eq.wthom}
\WHG X i =\WHG {M(X)} i.
\end{equation}

\ifShowNewSectionTwoPointTwoChanges
{ \color{blue}
\fi

\begin{rem}
Bondarko remarks in the introduction of \cite{Bo2} that the functor $h$ of Guillen and Navarro Aznar (\cite[Theorem 5.10]{GN}) is (essentially) $(t \circ M)(-)$. So our $\WHG - i$ would be equal to a ``Guillen-Navarro Aznar homology'' if it were defined as $H^i\Gamma \circ h$, analogously to Gillet-Soul{\'e} weight homology.
\end{rem}

\ifShowNewSectionTwoPointTwoChanges
} 
\fi

By \cite[\S2.2]{V1} (if resolution of singularities holds) or \cite{Ke} this gives rise to covariant functors 
\[
\WHG - i\;:\; \Sch / k \to \cA\quad (i\in \Z)
\]
satisfying the following properties:
\begin{itemize}
\item[$({\bf W}0)$](Nilpotent invariance)
For $X\in \Sch / k$ with its reduced scheme structure $X_{\red}\hookrightarrow X$, 
the map $\WHG {X_{\red}} i\to \WHG X i$ is an isomorphism for any $i\in \Z$.
\item[$({\bf W}1)$](Homotopy invariance)
For any $X\in \Sch / k$, the map
\mbox{$\WHG {X\times\bA^1} i\to \WHG X i$} induced by the projection is an isomorphism for any $i\in \Z$.
\item[$({\bf W}2)$](Mayer-Vietoris axiom)
For $X\in \Sch / k$ and an open covering $X=U\cup V$, one has a functorial long exact sequence
\[
\cdots\to \WHG {U\cap V} i \to \WHG {U} i \oplus \WHG {V} i \to \WHG {X} i \to \WHG {U\cap V} {i-1} \to \cdots
\]
\item[$({\bf W}3)$](Abstract blow-up)
Consider a cartesian diagram in $\Sch / k$:
\[
\xymatrix{
p^{-1}(Z) \ar[r]\ar[d] & X_Z \ar[d]\\
Z \ar[r] & X\\
}\]
where $p_Z:X_Z\to Z$ is proper and $Z\to X$ is a closed immersion such that $p^{-1}(X-Z) \to X-Z$ is an isomorphism.
Then one has a functorial long exact sequence
\[
\cdots\to \WHG {p_Z^{-1}(Z)} i \to \WHG {X_Z} i \oplus \WHG {Z} i \to \WHG {X} i \to \WHG {p_Z^{-1}(Z)} {i-1} \to \cdots
\]
\item[$({\bf W}4)$]
For $X\in \Sm / k$ smooth projective, we have
\[
\WHG X i =\left\{ \begin{array}{lr}
\Gamma(X) & \text{ for } i=0 \\
0 & \text{ for } i\ne 0 \end{array} \right. 
\]
\end{itemize}

By Theorem \ref{thm.N2}, if $\Gamma=\Gamma_A$ for an abelian group $A$,
we have the stronger property:

\begin{itemize}
\item[$({\bf W}5)$]
For $X\in \Sm / k$ smooth, we have
\[
\wHH i X A =\left\{ \begin{array}{lr}
A^{c(X)} & \text{ for } i=0 \\
0 & \text{ for } i\ne 0 \end{array} \right.  
\]
where $c(X)$ is the set of connected components of $X$.
\end{itemize}
\begin{itemize}
\item[$({\bf W}6)$]
For a proper $X\in \Sch / k$, we have $\WcHG {X} i= \WHG X i$.
\end{itemize}
\begin{itemize}
\item[$({\bf W}7)$] 
If the groups $\Gamma(X)$ are finitely generated for all $X$ smooth (e.g., $\Gamma = \Gamma_{\bZ / n}$) then the groups $\WHG X i$ are finitely generated for all $X \in \Sch / k$ if  
\end{itemize}
\medbreak

In the rest of this section, we assume that $k$ is a finite field with $p=\ch(k)$.
Recall (\cite[\S 8]{Ge2}) that since $k$ is finite, one can define Geisser's Kato-Suslin homology $H_i^{KS}(X, A)$ of $X \in \Sch / k$ with coefficients in an abelian group $A$ as the $i$th homology of the complex which we will denote by 
\begin{equation} \label{HKS}
R\Gamma^{KS}(X, A) =  \Tot \biggl ( C_\ast^X(k) \otimes A \to C_\ast^X(\overline{k}) \otimes A \stackrel{\id - \phi}{\to} C_\ast^X(\overline{k}) \otimes A \biggr )  
 \end{equation}
where $\overline{k}$ is the algebraic closure of the finite field $k$, the Frobenius of $k$ is denoted by $\phi$, and the presheaf of complexes $C_\ast^X(-)$ is the Suslin complex $\underline{C}_\ast(-)$ \cite[\S 7]{SV} of the presheaf $c_{\equi}(X / k)$ \cite{SVRelCyc}. Consider the property (W5) for Geisser's theory:
\begin{enumerate}
 \item[($\Sm$A)$_{KS}$] For $X \in \Sm / k$ smooth, we have
\[
H_i^{KS}(X, A) =\left\{ \begin{array}{lr}
A^{c(X)} & \text{ for } i=0 \\
0 & \text{ for } i\ne 0 \end{array} \right.  
\]
where $c(X)$ is the set of connected components of $X$.
\end{enumerate}

\begin{prop}[{\cite{Ge2}}] \label{Ge81}\ 
\begin{enumerate}
 \item \label{Ge81:torsion} If $A = \bZ / n$ with $n$ coprime to $p$, then ($\Sm$A)$_{KS}$ is satisfied.
 \item \label{Ge81:Q} If $A$ is $\bZ[1/p]$-module, then ($\Sm$A)$_{KS}$ holds if the following condition holds:
\begin{enumerate}
 \item[(P$_0$)] Rational Suslin homology is $\Sm$-acyclic. That is, for all $X$ smooth over a finite field, the rational Suslin homology groups $H_i^S(X, \bQ)$ are zero for $i \neq 0$.
\end{enumerate}
\end{enumerate}
\end{prop}

\begin{rem}\label{rem1.Ge81}
(P$_0$) is equivalent to the (part) of Parshin's conjecture (see \cite{Ge3}):
For $X$ smooth and proper over $k$, $\CH_0(X,i)$ is torsion for $i\not=0$.
\end{rem}
\def\MHA#1#2{H^M_{#2}(#1,A)}
\def\MHAet#1#2{H^{M,\et}_{#2}(#1,A)}

\begin{proof}
Prop.~\ref{Ge81}(2) is Geisser's \cite[Thm. 8.2]{Ge2}) (Indeed he shows also the converse implication for $A=\bQ$).
We give an alternative proof of Prop.~\ref{Ge81}.
Notice that $R\Gamma^{KS}(X, A)$ is by definition the cone of the morphism
\[ C_\ast^X(k) \otimes A \to \Tot \biggl ( C_\ast^X(\overline{k}) \otimes A \stackrel{\id - \phi}{\to} C_\ast^X(\overline{k}) \otimes A \biggr ). \]
The $i$th homology of the source is by definition $H_i^S(X, A)$. The $i$th homology of the target is what Geisser defines to be the \emph{arithmetic homology} and is denoted by $H_{i + 1}^{\textrm{ar}}(X, A)$ (\cite[\S 7.1]{Ge2}). 
So now ($\Sm$A)$_{KS}$ is equivalent to asking that the canonical morphisms 
$H_i^S(X, A) \to H_{i + 1}^{\textrm{ar}}(X, A)$ be isomorphisms for $i \neq 0$. 
By definition we have an exact sequence
\begin{equation}\label{eq1.Ge81}
0\to H_{i+1}^S(\Xb, A)_G \to H_{i + 1}^{\textrm{ar}}(X, A) \to H_{i}^S(\Xb, A)^G \to 0,
\end{equation}
where $\Xb=X\otimes_k \kb$ with an algebraic closure $\kb$ of $k$.
By \eqref{MHS} and Th.~\ref{mainthm.etreal} together with $({\bf W}5)$, we have isomorphisms for $X$ smooth,
\begin{equation}\label{eq2.Ge81}
\begin{aligned}
 & H_{i}^S(X, A) \cong \MHA {M(X)} i \cong \MHAet {M(X)} i,\\
 & H_{i}^S(\Xb, A) \cong \MHA {M(\Xb)} i \cong \MHAet {M(\Xb)} i.
\end{aligned}
\end{equation}
Letting $G_k$ be the absolute Galois group of $k$, we have a spectral sequence
\begin{equation}\label{galoisdescentss}
E^2_{i,j}=H^{-i}(G_k,\MHAet {M(\Xb)} j) \Rightarrow \MHRet {M(X)} {i+j},
\end{equation}
Assuming either that $A$ is torsion or that (P$_0$) holds, $\MHAet {M(\Xb)} j$ are torsion for $j\not=0$ and hence
\eqref{galoisdescentss} gives rise to  
\[
0\to {\MHAet {M(\Xb)} {i+1}}_G \to \MHA {M(X)} i \to {\MHAet {M(\Xb)} {i}}^G \to 0.
\]
In view of \eqref{eq2.Ge81}, this provides an exact sequence
\begin{equation*}\label{eq3.Ge81}
0\to H_{i+1}^S(\Xb, A)_G \to H_{i}^S(X, A) \to H_{i}^S(\Xb, A)^G \to 0.
\end{equation*}
Combined with \eqref{eq1.Ge81}, this proves the desired isomorphism
$H_i^S(X, A) \cong H_{i + 1}^{\textrm{ar}}(X, A)$.
\end{proof}

\begin{theo}\label{mainthm1}
Assume $k$ is finite and either $A$ is a torsion prime to $p$ or $A$ is a $\zpi$-module and (P$_0$) holds. Then there are isomorphisms
natural in $X \in \Sch / k$:
\begin{equation} \label{KSWHisos}
H_{i}^{KS}(X, A) \cong \wHH {i}  X A,
\end{equation}
where $\wHH {i}  X A$ is the weight homology from \S\ref{WHSch}.
\end{theo}

\begin{proof}
Since the presheaves $c_{\equi}(X / k)$ are functorial in $X \in \SCk$, the functor $R\Gamma^{KS}(-, A)$ is as well. In fact, using the total complex functor $\Tot$, the functor $R\Gamma^{KS}(-, A)$ can be extended to a functor on the homotopy category $K^b(\SCk)$ of bounded complexes in $\SCk$, inducing a triangulated functor $K^b(\SCk) \to K^b(\AB)$. For every $X \in \Sm / k$ and open subsets $U, V \subseteq X$, the functor $R\Gamma^{KS}(-, A)$ sends the complexes $(\bA^1_X \to X)$ and $(U \cap V \to U \oplus V \to U \cup V)$ to acyclic complexes since this is true for 
$C_\ast^{(-)}(k) \otimes A$ and $C_\ast^{(-)}(\overline{k}) \otimes A$. Hence  $R\Gamma^{KS}(-, A)$ factors via a triangulated functor (cf. \eqref{KDM})
\begin{equation*} \label{GKSFactor}
R\Gamma^{KS}(-, A): \DMeffgmkR \to K^b(\AB).
\end{equation*}

By the usual cdh/ldh descent arguments one can show a canonical isomorphism
\[
H_{i}^{KS}(X, A) \cong H^{-i}(R\Gamma^{KS}(M(X), A))\qfor X \in \Sch / k.
\]
By Prop.~\ref{Ge81}
the homological functor $R\Gamma^{KS}(-, A)$ is $\SP$-acyclic. Hence it now follows from the equivalence of Theorem~\ref{equa:equivalences} that $H^{-i}R\Gamma^{KS}(-, A)$ is canonically isomorphic to the homological functor $\wHH {i}  - A$ on $\DMeffgmkR$. 
%
%
\end{proof}



\section{Motivic homology and \'etale motivic homology} \label{sec:homology}

We continue with $k$ a perfect field of exponential characteristic $p$. We let $\DMeffkRet$ denote the {\'e}tale version of $\DMeffkR$ (i.e., constructed with {\'e}tale sheaves with transfers instead of Nisnevich ones) and
\begin{equation} \label{alphaUpStar}
\alpha^*: \DMeffkR \to \DMeffkRet
\end{equation}
the canonical functor induced by {\'e}tale sheafification. The functor $\alpha^*$ admits a right adjoint which is denoted by $\alpha_*$.

\begin{lem}\label{alphafullyfaith}
The right adjoint $\alpha_*$ to $\alpha^*$ (Equation~\eqref{alphaUpStar}) is fully faithful.
\end{lem}
\begin{proof}
It suffices to show that the counit map $\alpha^*\alpha_* C\to C$ is an isomorphism for $C\in \DMeffkRet$. 
We may assume $C=F[0]$ for a single \'etale sheaf $F$.
Then the assertion is true for tautological reasons: Indeed, $H^0(\alpha_* F)$ is just $F$ where we only remember that it is a Nisnevich sheaf so that $\alpha^*H^0(\alpha_* F) \to F$ is an isomorphism. 
On the other hand, for $q> 0$, the \'etale sheaf associated to $H^q(\alpha_* F)$ is $0$, i.e. 
$\alpha^*H^q(\alpha_* F)=0$. This completes the proof of the lemma.
\end{proof}

The motivic homology and {\'e}tale motivic homology with $R$-coefficients of an object $C$ in $\DMeffkR$ are defined as
\begin{align}
\label{MHR} \MHR C i &= \hom_{\DMeffkR}(R[i], C), \quad \textrm{ and } \\ 
\MHRet C i &= \hom_{\DMeffkR}(R[i], \alpha_*\alpha^*C)
\end{align}
where $R[i] = M(k)[i]$ is a shift of the unit for the tensor structures. 

\medbreak

The unit $\id \to \alpha_*\alpha^*$ of the adjunction $(\alpha^*, \alpha_*)$ induces a natural map 
\[
\alpha^*: \MHR C i \to \MHRet C i 
\]
which is functorial in $C\in\DMeffkR$. 

%
For $C\in \DMeffgmkR$ we put
\[
\WHL C i = \wH {C[i]} \Lam,
\]
where $\wH -\Lam : \DMeffgmkR \to \AB$ is the weight homology functor from Definition \ref{def.weighthomology} for $A=\Lam$.

\begin{theo}\label{mainthm.etreal} \ 
Let $k$ be a field of exponential characteristic $p$ and let $n$ be an integer prime to $p$. Suppose either $R = \Lambda = \ZZ / n$, or $R = \bZ[1/p]$ and $\Lambda = \bQ/\bZ[1/p]$.
\begin{itemize}
\item[(1)]
If $k$ is algebraically closed, $\alpha^*: \MHR C i \to \MHRet C i 
$ is an isomorphism.
\item[(2)]
If $k$ is finite and $C\in \DMeffgmkR$, there is a canonical long exact sequence
 \[
\cdots \to \MHR C i \rmapo{\alpha^*} \MHRet C i  \to \WHL C {i+1} \to \MHR C {i-1} \to \cdots
\]
which is functorial in $C$. 
\end{itemize}
\end{theo}

The proof of the theorem will be given in the following sections.
\medbreak

Take $R=\Lam=\nz$ with $(n,p)=1$.
Then for $C=M(X)$ (resp. $M^c(X)$), the groups $\MHnzet C i $  are identified with the dual of \'etale cohomology
(resp. \'etale cohomology with compact support) by Proposition \ref{prop4.etreal} and Lemma \ref{lem.ethom} below. 
Thus \eqref{MHS} and \eqref{MHhigherChow} and Theorem \ref{mainthm.etreal}(2) imply the following:

\begin{cor}\label{cor.etreal}
Assume $k$ is finite. For an integer $n$ prime to $p$, there are long exact sequences functorial in $X\in \Sch/k$: 
\begin{multline}\label{HSetcoh}
\cdots \to H^S_i(X,\nz) \to H^{i+1}(X_{\et},\nz)^* \to \WHnz X {i+1} \\
\to H^S_{i-1}(X,\nz) \to \cdots\;,
\end{multline}
\begin{multline}\label{CHetcoh}
\cdots \to \CH_0(X,i;\nz) \to H^{i+1}_c(X_{\et},\nz)^* \to \WcHnz X {i+1} \\
\to \CH_0(X,i-1;\nz) \to \cdots\;,
\end{multline}
where $\WHnz X i$ and $\WcHnz X i$ are the weight homology and the weight homology with compact support
(cf. \eqref{eq.wthom} and \eqref{eq.wthomc}). 
In particular, if $X$ is smooth, $({\bf W}5)$ implies
canonical isomorphisms
\begin{equation}\label{HSetcohsmooth}
H^S_i(X,\nz) \cong H^{i+1}(X_{\et},\nz)^*\qfor i\in \bZ.
\end{equation}
\end{cor}


\begin{rem}\label{rem.Geisser}
Under the assumption of Cor.~\ref{cor.etreal}, Geisser \cite{Ge2} shows an isomorphism
\begin{equation}\label{aretisom}
H_{i}^{\textrm{ar}}(X,\nz)\cong H^{i}(X_{\et},\nz)^*,
\end{equation}
where $H_{i}^{\textrm{ar}}(X,\nz)$ is the arithmetic homology (see the proof of Prop.~\ref{Ge81}).
Under \eqref{aretisom} and the identification of Theorem~\ref{mainthm1}, the long exact sequence \eqref{HSetcoh} is identified with the long exact sequence from \cite[\S 8]{Ge2}:%

\begin{multline}\label{HSarKS}
\cdots \to H^S_i(X,\nz) \to H_{i+1}^{\textrm{ar}}(X,\nz) \to H^{KS}_{i + 1}(X, \nz) \\
\to H^S_{i-1}(X,\nz) \to \cdots
\end{multline}
%
%
%
\end{rem}
\medbreak

As a consequence of the long exact sequence \eqref{HSetcoh} and property (W7) we obtain an alternative proof of the prime-to-$p$ part of \cite[Thm. 6.1]{Ge2}:

\begin{theo}[{cf. \cite[Thm. 6.1]{Ge2}}]
For $k$ a finite field and $n$ prime to the characteristic, 
the groups $H^S_i(X, \ZZ / n \ZZ)$ are finitely generated for all $i$ and all $X \in \Sch / k$.
%
\end{theo}

\begin{rem}\label{rem.WH}
We have a natural isomorphism
\[
H^{1}(X_{\et},\nz)^*\cong \pi_1^{\ab}(X)/n,
\]
where $\pi_1^{\ab}(X)$ is the abelian fundamental group of $X$.
Hence \eqref{HSetcoh} gives an exact sequence
\begin{equation}\label{rhoX}
\WHnz X {2} \to  H^S_0(X,\nz) \rmapo{\rho_X} \pi_1^{\ab}(X)/n \to \WHnz X {1},
\end{equation}
Functoriality implies that for a closed point $x\in X$, the composite of $\rho_X$ with
\[\Z = H^S_0(X,\Z) \to H^S_0(X,\nz)\]
sends $1\in \Z$ to the Frobenius element at $x$ in $\pi_1^{\ab}(X)/n$.  
Thus \eqref{rhoX} and \eqref{HSetcohsmooth} provide a generalisation of tame class field theory for
smooth schemes over a finite field studied in \cite{ScSp}. 
A generalisation of \cite{ScSp} in a different direction is studied in \cite{GeSc}.
\end{rem}

\begin{exam}\label{WHexam2}
Here is an example of computation of $\WHL X {i}$ for $X\in \Sch/k$.
Assume that there is a cartesian diagram:
\[
\xymatrix{
p^{-1}(Z) \ar[r]\ar[d] & X_Z \ar[d]\\
Z \ar[r] & X\\
}\]
where $p_Z:X_Z\to Z$ is proper and $Z\to X$ is a closed immersion such that $X_Z$ is smooth, 
$p^{-1}(X-Z) \to X-Z$ is an isomorphism and the reduced part $E$ of $p^{-1}(Z)$ is a simple normal crossing divisor on $X_Z$. 
Assume additionally that $Z_{\red}$ is proper smooth (e.g. $\dim(Z)=0$).
Then, using $({\bf W}0)$, $({\bf W}3)$, $({\bf W}4)$, $({\bf Wc}1)$ and Example \ref{WHexam1}, 
we get an isomorphism:
\[
\WHL X {a+1} \cong H_{a}(\Gamma_E(\Lam))\qfor a\geq 1. 
\]
\[
\WHL X {1} \cong \Ker\big(\Lam^{c(E)} \to \Lam^{c(Z)}\oplus\Lam^{c(X_Z)}\big),
\]
where under the same notation as Example \ref{WHexam1}, $\Gamma_E(\Lam)$ is the complex
$$ 
 \cdots\to \Lam^{c(E^{[a]})} \rmapo{\partial_a} \Lam^{c(E^{[a-1]})} 
\rmapo{\partial_{a-1}} \cdots \rmapo{\partial_1} \Lam^{c(E^{[0]})}. 
$$
\end{exam}


\section{\'Etale cycle map for motivic homology}\label{etcyclemap}

In this section we prove the following theorem.

\begin{theo}\label{etcycl-propersmooth}
Let $X$ be smooth over $k$ and $n$ a positive integer invertible in $k$.  For an integer $r \geq 0$, the map
\[\alpha^*: \MHnz {M^c(X)(r)} i \to \MHnzet {M^c(X)(r)} i \]
is an isomorphism under one of the following conditions:
\begin{itemize}
\item[$(i)$] \label{etcycl-propersmooth:rational}
$k$ is algebraically closed.
\item[$(ii)$]
$k$ is finite and $r>0$.
\item[$(iii)$]
$k$ is finite, $X$ is proper over $k$, and $r=0$ and $i\not=-1$. 
\end{itemize}
\end{theo}
\medbreak

We need the following result:

\begin{prop}\label{prop.etreal} If $k$ has finite cohomological dimension then:
\begin{itemize}
\item[(1)]
If $R$ is a $\bQ$-algebra, $\alpha^*$ in Equation~\eqref{alphaUpStar} is an equivalence of $\otimes$-triangulated categories.
\item[(2)]
If $p=\ch(k)>0$ and $R$ is annhilated by a power of $p$, then $\DMeffkRet=0$.
\item[(3)]
For $n$ prime to $p$, let $\Dknz$ be the derived category of sheaves of $\nz$-modules over the small {\'e}tale site $Et(k)$ of $\Spec (k)$.
Then the functor
\begin{equation*}
\Phi^{\et}: \DMeffknzet \to \Dknz
\end{equation*}
induced by the association $\cF\mapsto \cF(\bar k)$ for an \'etale sheaf $\cF$ on $\Sm / k$, 
is an equivalence of $\otimes$-triangulated categories. 
\end{itemize}
\end{prop}
\begin{proof}
In the bounded above case, this is due to Voevodsky (\cite[Prop. 3.3.2 and 3.3.3]{V1}).
The general case is shown by the same argument.
\end{proof}

In the appendix we show that the functors $\Phi^{\et}, \alpha^*, M,$ and $M^c$ are all compatible in the way one would expect:

\begin{prop}\label{prop3.etreal}
Suppose that $k$ is algebraically closed. For $X\in \Sch / k$ with structural morphism $\pi:X\to \Spec(k)$, there are canonical functorial isomorphisms
\begin{align*}
\rHom(\Phi^{\et}\alpha^* M(X), \nz) &= R\pi_{*}(\nz)_X, \qaq \\
\rHom(\Phi^{\et}\alpha^* M^c(X)), \nz) &= R\pi_{!}(\nz)_X.
\end{align*}
\end{prop}
\medbreak

By Propositions~\ref{prop.etreal} and \ref{prop3.etreal}, we get the following.

\begin{prop}\label{prop4.etreal}
For $X\in \Sch / k$ and $i, r \in \Z$, there are natural isomorphisms
\begin{equation*}\label{etcyclemap1}
\begin{aligned}
&\MHnzet {M(X)(r)} i  \cong \Hnztet X i {r},\\
&\MHnzet {M^c(X)(r)} i \cong \Hcnztet X i {r}.\\
\end{aligned}
\end{equation*}
where 
\begin{equation}\label{etmothomX}
\begin{aligned}
&\Hnztet X i r =\Hom_{\Dknz}(R\pi_{*}(\nz)_X,\mu_n^{\otimes r}[-i]),\\
&\Hcnztet X i r =\Hom_{\Dknz}(R\pi_{!}(\nz)_X,\mu_n^{\otimes r}[-i]).
\end{aligned}
\end{equation}
called the \emph{\'etale homology} and \emph{\'etale homology with compact support} of $X$ with $\nz$-coefficients  respectively. %
\end{prop}
\medbreak

We now ``convert'' \'etale homology into cohomology:

\begin{lem}\label{lem.ethom}
Let $X\in \Sch/k$ and let $n$ be a positive integer invertible in $k$.
\begin{itemize}
\item[(1)]
If $X$ is smooth over $k$ of pure dimension $d$, there are canonical isomorphisms
\[
\begin{aligned}
& \Hnztet X i r \cong H^{2d-i}_c(X_{\et},\mu_n^{\otimes d+r}),\\
& \Hcnztet X i r \cong H^{2d-i}(X_{\et},\mu_n^{\otimes d+r}).\\
\end{aligned}
\]
\item[(2)]
If $k$ is algebraically closed there are canonical isomorphisms 
\begin{equation*} \label{etaleHomology}
\begin{aligned}
&\Hnztet X i r \cong H^{i}(X_{\et},\mu_n^{\otimes -r})^*,\\
&\Hcnztet X i r\cong H^{i}_c(X_{\et},\mu_n^{\otimes -r})^*. \\
\end{aligned}
\end{equation*}
\item[(3)]
If $k$ is finite there are canonical isomorphisms 
\begin{equation*}
\begin{aligned}
&\Hnztet X i r \cong H^{i+1}(X_{\et},\mu_n^{\otimes -r})^*,\\
&\Hcnztet X i r\cong H^{i+1}_c(X_{\et},\mu_n^{\otimes -r})^*.\\
\end{aligned}
\end{equation*}
\end{itemize}
\end{lem}
\begin{proof}
By duality, we have natural isomorphisms
\begin{equation*}
\begin{aligned}
&\uHom_{\Dknz}(R\pi_{*}(\nz)_X,\nz)\cong R\pi_{!}R\pi^!\nz,\\
&\uHom_{\Dknz}(R\pi_{!}(\nz)_X,\nz)\cong R\pi_{*}R\pi^!\nz.\\
\end{aligned}
\end{equation*}
where $\uHom_{\Dknz}$ denotes the inner $\Hom$ in $\Dknz$.
For $X\in \Sm / k$ of pure dimension $d$, we have $R\pi^!\nz=\mu_n^{\otimes d}[2d]$ and (1) and (2) follow from this.
If $k$ is finite, (3) follows from the Tate duality for Galois cohomology of finite fields.
\end{proof}
\bigskip

\subsection{Proof of Theorem~\ref{etcycl-propersmooth}} \label{pf51}

Let $X$ be smooth of pure dimension $d$ over a perfect field $k$.  It is shown in \cite{Iv2} that the following diagram is commutative:
\[
\xymatrix{
\MHnz {M^c(X)(r)} i \ar[r]^{\alpha^*} \ar[d]_{\simeq}^{(*1)} & \MHnzet {M^c(X)(r)} i  \ar[d]_{\simeq}^{(*2)}\\
\CH_{-r}(X,i+2r;\nz) \ar[d]^{=} &  H^{2d-i}(X_\et,\mu_n^{\otimes d+r})^* \\
\CH^{d+r}(X,i+2r;\nz) \ar[ru]^{cl_X},\\
}\]
where $cl_X$ is the \'etale cycle map defined by Geisser-Levine \cite{GL}, and 
$(*1)$ comes from Equation~\eqref{MHhigherChow} and $(*2)$ from Equation~\eqref{etcyclemap1} and Lemma~\ref{lem.ethom}(1).

Similarly, when $k$ is algebraically closed, in \cite{Ke2} it is shown that (for $r \geq 0$) this same diagram is commutative but where now $cl_X$ is Suslin's isomorphism from \cite{Su} (cf. \cite[Thm. 5.6.1]{Ke}).

Hence Theorem~\ref{etcycl-propersmooth} follows from the Theorem~\ref{etcyclhigherChow} below.

\begin{theo}\label{etcyclhigherChow}
Under the assumption of Theorem \ref{etcycl-propersmooth}, the Geisser-Levine cycle map
\[
cl^{d+r,q}_X: \CH^{d+r}(X,q;\nz) \to H^{2(d+r)-q}(X_\et,\nz(d+r))^*
\]
are isomorphisms for all $q\in \bZ$.
\end{theo}
\begin{proof}
The case $(ii)$ (resp. $(iii)$) is shown in \cite[Lem.6.2]{JS} (resp. \cite[Thm. (9.3]{KeS}).
The case $(i)$ is a consequence of \cite{Su} and \cite[Theo. 5.6.1]{Ke}. 
Here we give an alternative proof of $(i)$ following the proof of \cite[Lem.6.2]{JS}.
It is simpler but depends on the Beilinson-Lichtenbaum conjecture
due to Suslin-Voevodsky \cite{SVBlochKato} and Geisser-Levine \cite[Thm.1.5]{GL} 
which relies on the Bloch-Kato conjecture proved by Rost-Voevodsky.

By the localization theorem for higher Chow groups (\cite{B} and \cite{L}),
we have the niveau spectral sequence
$$
{^{CH}E}^1_{a,b}= \underset{x\in X_a}{\bigoplus} \CH^{a+r}(x,a+b;\nz)
\Rightarrow \CH^{d+r}(X,a+b;\nz).
$$
By the purity for \'etale cohomology, we have the niveau spectral sequence
$$
{^{\et}E}^1_{a,b}= \underset{x\in X_a}{\bigoplus} H^{a-b+2r}_{\et}(x,\nz(a+r))
\Rightarrow H^{2(d+r)-a-b}_{\et}(X,\nz(d+r)).
$$
The cyle map $cl^{d+r,a+b}_X$ preserves the induced filtrations and induces maps on $E^\infty_{a,b}$ compatible with
the map on $E^1$-terms induced by the cycle maps for $x\in X_a$:
$$
cl_x^{a+r,a+b}\;:\; \CH^{a+r}(x,a+b;\nz) \to H^{a-b+2r}_{\et}(x,\nz(a+r)).
$$
By the Beilinson-Lichtenbaum conjecture, $cl_x^{a+r,a+b}$ is an isomorphism if $b\geq r$.
On the other hand, we have $ {^{CH}E}^1_{a,b}=0$ for $b<r$ since $\CH^{i}(L,j;\nz)=0$ for a field $L$ and integers $i>j$, and we have $^{\et}E^1_{a,b}=0$ for $b < 2r$ since $cd(\kappa(x))=a$ for $x\in X_a$ by the assumption that 
$k$ is algebraically closed.
Noting $r\geq 0$, these imply that the cycle maps induce an isomorphism of the spectral sequences and hence
the desired isomorphism of Theorem \ref{etcyclhigherChow}.
\end{proof}



\section{Reduced motives and co-{\'e}tale motives} \label{effmotive}

In case $k$ is finite, the weight homology functor in Definition \ref{def.weighthomology} is compared 
with a different homology functor which is defined as the homology of the ``co-{\'e}tale'' part of motives
(see Theorem \ref{thm.wHTheta}(2)).
To introduce this, we need some preliminaries. In this section we still work over a perfect field $k$.

\subsection{A triangulated lemma}

We will apply the following lemma twice; in Equation~\eqref{defi:theta} and Equation~\eqref{defi:red}.

\begin{lem}\label{lem.tr} 
Suppose that $B: \cT \to \cS$ is a fully faithful exact functor of triangulated categories which admits a left adjoint $A: \cS \to \cT$. Then there is a triangulated endofunctor
\[C:\cS \to \cS\]
with natural transformations $c: C \to id_{\cS}$ and $d: BA \to C[1]$ such that for any object $X\in \cS$,
 \[ C(X) \by{c_X} X \by{\eta_X} BA(X) \by{d_X} C(X)[1] \]
is a distinguished triangle of $\cS$, where $\eta_X$ is the unit map of the adjunction.
If small direct sums are representable in $\cS,\cT$ and $A,B$ commute with them, so does $C$.
\end{lem}

\begin{proof} Take $X\in \cS$ and choose an object $C\in \cS$ fitting in a distinguished triangle
 \[C \to X\by{\eta_X} BA(X)\by{+1}\]
 We want to show that $C$ is determined by $X$ uniquely up to unique isomorphism.
For this it suffices to see that, for any $Y\in \cS$, one has
$\Hom_{\cS}(C,BA(Y))=0$.
By adjunction, we are reduced to showing that $A(C)=0$, equivalently the map $A(\eta_X):A(X)\to ABA(X)$ is an isomorphism.
This holds since $A(\eta_X)$ is split by the map $\epsilon_{A(X)}:ABA(X)\to A(X)$, 
where $\epsilon$ is the counit of the adjunction, and the latter is an isomorphism by the full faithfulness of $B$. 
The last statement is also obvious. 
\end{proof}

The first application of Lemma~\ref{lem.tr} is to the adjunction
\[ (\alpha^*, \alpha_*): \DMeffkR \rightleftarrows \DMeffkRet, \]
we get an endomorphism
\begin{equation} \label{defi:theta}
\Theta: \DMeffkR \to \DMeffkR \ ;\qquad  C \to \Theta(C)
\end{equation}
equipped with natural transformations $\alpha_*\alpha^*[-1] \to \Theta \to \id $ such that 
\begin{equation} \label{Thetadt}
 \Theta (C) \to C \to  \alpha_*\alpha^* C \by{+1}
\end{equation}
is a distinguished triangle for every $C\in \DMeffkR$. 
One could call $\Theta (C)$ the \emph{co-{\'e}tale part} of $C$ since $\hom_{\DMeffkR}(\Theta(C), C') = 0$ for all $C'$ in the image of $\alpha_*:\DMeffkRet\to \DMeffkR $. Notice that a consequence of Proposition \ref{prop.etreal}(1) and (2) is that $\Theta(C)$ is necessarily torsion prime to $\ch(k)$:

\begin{lem}\label{lemm:Thetator}
For every $C \in \DMeffkR$ the canonical morphism
\begin{equation}\label{Thetator}
\underset{n}{\text{hocolim}} \uHom(\Z/n,\Theta(C))\iso \Theta(C),
\end{equation}
is an isomorphism where $\uHom$ is the inner $\Hom$ in $\DMeffkR$ and $n$ ranges over all positive integers invertible in $k$.
\end{lem}

Note also that 
\begin{equation}\label{Thetator2}
\Theta(\uHom(\Z/n,C))\simeq \uHom(\Z/n,\Theta(C)).
\end{equation}
\medbreak

In view of \eqref{Thetadt} Theorem \ref{mainthm.etreal} is a consequence of the following.

\begin{theo}\label{thm.wHTheta} 
Let $C\in \DMeffkR$. Suppose either $R = \Lambda = \ZZ / n$, or $R = \bZ[1/p]$ and $\Lambda = \bQ/\bZ[1/p]$.
\begin{itemize}
\item[(1)] 
If $k$ is algebraically closed, $\MHR {\Theta(C))} i=0$ for $i\in \bZ$.
\item[(2)] 
If $k$ is finite, there are canonical isomorphisms for $C\in \DMeffgmkR$: 
\[
\MHR {\Theta(C))} i \isom \WHL C {i+1}\qfor i\in \bZ,
\] 
which are natural in $C$.
\end{itemize}
\end{theo}

\subsection{Reduced motives} \label{reduced motives}
For $r\ge 0$, let $d_{\le r}\DMeffkR$ be the localizing subcategory of $\DMeffkR$ generated by motives of smooth varieties of dimension $\le r$ (cf. \cite[\S 3.4]{V1}). We have a natural equivalence of categories
\begin{equation*}\label{d0DMA}
d_{\le 0}\DMeffkR \cong D(\cA_R),
\end{equation*}
where 
$\cA_R$  is the category of additive contravariant functors on the category of permutation $R[G_k]$-modules 
\cite[Prop. 3.4.1]{V1}. 
\medbreak

In \cite{abv}, Ayoub and Barbieri-Viale construct the left adjoint 
\begin{equation}\label{Lpi0}
L\pi_0: \DMeffkR \to d_{\le 0}\DMeffkR
\end{equation}
to the inclusion functor $d_{\le 0} \DMeffkR\to \DMeffkR$. It satisfies
\begin{equation}\label{pi0M(X)}
L\pi_0(M(X)) = M(\pi_0(X)) \qfor X\in \Sm,
\end{equation}
where $\pi_0(X)$ is the ``scheme of constants" of $X$, i.e.
\[
\pi_0(X) =\underset{i\in I}{\coprod} \;\Spec(L_i),
\]
where $\{X_i\}_{i\in I}$ are the connected components of $X$ and $L_i$ is the algebraic closure of $k$ in the function
field of $X_i$. 
The functor $L\pi_0$ commutes with small direct sums. By Lemma \ref{lem.tr} there is a triangulated endofunctor
\begin{equation} \label{defi:red}
(-)_{\red} : \DMeffkR \to \DMeffkR \ ;\qquad  C \to C_\red
\end{equation}
with natural transformations $L\pi_0[-1] \to (-)_\red \stackrel{\iota}{\to} \id$ such that 
\begin{equation} \label{redTri}
C_\red \rmapo \iota C \to  L\pi_0(C)\by{+1}
\end{equation}
is a distinguished triangle for every $C\in \DMeffkR$. 

\begin{defn} \label{defi:reduced}
We call $C_\red$ the \emph{reduced part} of $C$.
We may say that $C$ is \emph{reduced} if the canonical morphism $C_\red \to C$ is an isomorphism, or equivalently if $L\pi_0C = 0$.
\end{defn}

\subsection{Proof of Theorem \ref{thm.wHTheta}}

We start with the following:

\begin{claim}\label{claim1.etreal} 
Let $C\in \DMeffkR$.
\begin{itemize}
\item[(1)]
If $k$ is algebraically closed, we have $\MHR {\Theta(C)} i =0$ for $i\in \Z$.
\item[(2)]
If $k$ is finite, we have $\MHR {\Theta(C_{\red})} i =0$ for $i\in \Z$.
\end{itemize}
\end{claim}
\begin{proof}
\emph{Reduction to $R=\nz$ and $C=M(X)$ with $X$ smooth and proper over $k$:} 
Since $M(k)$ is a compact object of $\DMeffkR$, the functor 
\[
\MHR - i : \DMeffkR \to \AB
\]
commutes with infinite direct sums, hence with $\text{hocolim}_n$.
Therefore, in light of the isomorphisms
\[
\Theta(C) \iso \underset{n}{\text{hocolim}} \uHom(\Z/n,\Theta(C)) \iso 
\underset{n}{\text{hocolim}} \ \Theta(\uHom(\Z/n,C))
\]
of Lemma~\ref{Thetator} and Equation~\ref{Thetator2} we may assume $C$ is killed by some $n$ and $R=\nz$.

\emph{Reduction to $C = M(X)$:} Since  $\alpha^*$, $\alpha_*$ and $L\pi_0$ commute with infinite direct sums, the statement of the claim is true 
for a direct sum $\bigoplus_{\alpha\in A} C_\alpha$ if it is true for all $C_\alpha$. 
Moreover, if $C'\to C\to C''\by{+1}$ is an exact triangle in $DM^\eff(k,\Z/n)$ and the statement of the claim is true for $C'$ and $C''$, then it is true for $C$.
By Gabber's refinement of de Jong's theorem \cite{Il}, motives of smooth projective varieties are dense in $\DMeffknz$ in the sense that the smallest triangulated subcategory containing their shifts and stable under infinite direct sums is equal to $\DMeffknz$ (see \cite[Prop.5.5.3]{Ke}). Thus it suffices to show Claim \ref{claim1.etreal} 
for $R=\nz$ and $C=M(X)$, where $X$ is a smooth and proper over $k$.

\emph{Proof of Claim~\ref{claim1.etreal}(1):} Since we have reduced to the case $C=M(X)$ with $X$ smooth proper over $k$ and $R=\nz$, Claim~\ref{claim1.etreal}(1) follows from Theorem~\ref{etcycl-propersmooth}.

\emph{Proof of Claim~\ref{claim1.etreal}(2):} Now we prove the second assertion of the claim, where $k$ is finite and $C=M(X)$ with $X$ smooth proper over $k$.
We have a commutative diagram
\[
\xymatrix@C=6pt{
\cdots \ar[r] &\MHnz {M(X)_{\red}} i\ar[r] \ar[d]^{\alpha^*} &\MHnz {M(X)} i\ar[r] \ar[d]^{\alpha^*}
&\MHnz {L\pi_0 M(X)} i\ar[r] \ar[d]^{\alpha^*}&\cdots \\
\cdots \ar[r] &\MHnzet {M(X)_{\red}} i\ar[r]  &\MHnzet {M(X)} i\ar[r] 
&\MHnzet {L\pi_0 M(X)} i\ar[r] &\cdots \\
}\]
We want to show that $\alpha^*$ on the left hand side is an isomorphism.
By Theorem \ref{etcycl-propersmooth} and Equation~\eqref{pi0M(X)}, $\alpha^*$ in the middle and on the right hand side are 
isomorphisms for $i\not=-1$. By Equation~\eqref{pi0M(X)} and Equation~\eqref{MHhigherChow} we have
\[
\MHnz {M(X)} {-1} = \MHnz {L\pi_0 M(X)} {-1}=0.
\]
Hence we are reduced to showing that the map
\[
\MHnzet {M(X)} {-1} \to \MHnzet {L\pi_0 M(X)} {-1}
\]
is an isomorphism. By \eqref{etcyclemap1} and Lemma \ref{lem.ethom}, this follows from the fact that 
the trace map
\[
H^{2d+1}(X_{\et},\mu_n^{\otimes d}) \to H^{1}(\pi_0(X)_{\et},\nz)
\]
is an isomorphism. This completes the proof of Claim \ref{claim1.etreal}.
\end{proof}
\medbreak

It remains to show Theorem \ref{thm.wHTheta}(2). Assume $k$ is finite.
By definition we have a distinguished triangle
\[
\Theta(C_{\red}) \to \Theta(C)\to \Theta(L\pi_0(C)) \by {+1}.
\]
Hence Claim \ref{claim1.etreal}(2) implies 
\begin{equation} \label{piNotNec}
 \MHR {\Theta(C)} i \cong \MHR {\Theta(L\pi_0(C)))} i
 \end{equation}
  hence we are reduced to showing 
\begin{equation} \label{piNotwHomology}
\MHR {\Theta(L\pi_0(C)))} i\cong  \WHL C {i+1}.
 \end{equation}
By Theorem \ref{equivalenceHomologicalFunctors} this follows from the following.

\begin{claim}\label{claim2.etreal} 
Suppose $k$ is finite.
For $X\in \Sm$, 
\begin{equation*}\label{WHMX}
\MHR {\Theta(L\pi_0(M(X)))} i =\left\{ \begin{array}{lr}
\Lam^{c(X)} & \text{ for } i=-1 \\
0 & \text{ for } i\ne -1 \end{array} \right.  
\end{equation*}
where $c(X)$ is the set of connected components of $X$.
\end{claim}
\begin{proof}
Writing  
\begin{equation}\label{HCTR}
\HCTR C i = \MHR {\Theta(L\pi_0(C)))} i\qfor C\in \DMeffgmkR,
\end{equation}
this gives a homological functor on $\DMeffgmkR$. Note that by Equations~\eqref{Thetator} and \eqref{Thetator2},
we have for $C\in \DMeffkR$,
\begin{equation}\label{Thetator3}
\Theta(C) =\Theta(C_{\tor}) \qwith C_{\tor} = \underset{n}{\text{hocolim}} \uHom(\Z/n,C).
\end{equation}
We may assume $X$ is geometrically connected over a field $L$ finite over $k$. By Equation~\eqref{pi0M(X)}, we have
$L\pi_0(M(X)) =M(\Spec(L))$. We easily see $ M(\Spec(L))_{\tor} = \Lam_L[-1]$, where
\[\Lam_L = \left\{ \begin{array}{lr}
M(\Spec(L))\otimes \Q/\Z[1/p] & \text{ if } R=\bZ[1/p], \\
M(\Spec(L)) & \text{ if } R=\nz. \end{array} \right.  \]

Hence we get a distinguished triangle
\[
\Theta(L\pi_0(M(X))) \to \Lam_L[-1]\to \alpha_*\alpha^* \Lam_L[-1] \by{+1}
\]
which induces a long exact sequence
\begin{align*}
 \cdots\to \HCTR {M(X)} {i+1} \to \MHR {\Spec(L)} {i} \rmapo{\alpha^*} \MHRet {\Spec(L)} {i} \\
\to \HCTR {M(X)} i  \to \cdots.
\end{align*}
By Theorem \ref{etcycl-propersmooth}, $\alpha^*$ is an isomorphism for $i\ne -1$.
For $i=-1$, we have $\MHR {\Spec(L)} {-1}=0$ while by Lemma~\ref{lem.ethom} we have 
\[
\MHRet {\Spec(L)} {-1} \cong H^1(\Spec(k)_{\et},\Lambda)\cong \Lam,
\]
where the second isomorphism is the trace map in the Tate duality for Galois cohomology of finite fields. 
This completes the proof of Claim \ref{claim2.etreal}.
\end{proof}

\appendix


\section{{\'E}tale realisations of non-smooth motives} \label{etalerealisationnonsmooth}

We now show the compatibility of the functors 
\[ \alpha^*: \DMeffkR \to \DMeffkRet \]
\[ \Phi^{\et}: \DMeffknzet \to \Dknz \]
\begin{align*} 
M: \Sch / k \to &\DMeffkR \ ;  &X \mapsto M(X)  \\
M^c: \Sch^{\mathrm{prop}} / k \to &\DMeffkR \ ;  &X \mapsto M^c(X) 
\end{align*}
from Equation~\eqref{alphaUpStar}, Proposition~\ref{prop.etreal}, and Equations~\eqref{nonsmoothmotives} and \eqref{motiveCompact} respectively. 

We begin by recalling results of Suslin-Voevodsky that we will use. Let \mbox{$\nz \to \cI^\ast$} be an injective resolution of the constant $\qfh$-sheaf associated to $\nz$ in $\Shv_{\qfh}(\Sch / k)$. Recall that every $\qfh$-sheaf has a canonical structure of transfers (Yoneda's Lemma together with \cite[Thm. 4.2.12(1)]{SVRelCyc}). So the two categories $\Shv_{\qfh}(\Ck)$ and $\Shv_{\qfh}(\Sch / k)$ are canonically equivalent. 

For a scheme $X \in \Sch / k$ let $[X]$ denote its corresponding object in $\Ck$ as well as the presheaf with transfers on $\Ck$ that it represents. The two natural transformations
\[ \eta: \id \to \underline{C}_\ast(-), \qquad \qquad \epsilon: \nz \to \cI^\ast \]
in $\Comp(\Shv_{\qfh}(\Ck))$ together with the global sections functor
\[ \Gamma(k, -): \Shv_{\qfh}(\Ck) \to \AB \]
induce canonical natural transformations
\begin{align}
\hom(C_\ast([X]), \nz) \stackrel{\epsilon}{\to} &\hom(C_\ast([X]), \cI^\ast(k)) \label{qiAug} \\  
\hom(\underline{C}_\ast([X]), \cI^\ast) \stackrel{\Gamma(k, -)}{\to} &\hom(C_\ast([X]), \cI^\ast(k))  \label{qiSuslin} \\
\hom(\underline{C}_\ast([X]), \cI^\ast) \stackrel{\eta}{\to} &\hom([X], \cI^\ast) \cong \cI^\ast(X) \label{qiEtale}
\end{align}
of complexes of abelian groups.

\begin{theo}[Suslin-Voevodsky] \label{SVrigidity}
Suppose that the field $k$ is algebraically closed, $n$ is invertible in $k$, and $ X\in \Sch / k$. The complex $\cI^\ast$ and the sheaves $\cI^i$ are all {\'e}tale acyclic. That is, one has
\[ a_{\et}(\underline{H}^i(\cI^\ast)(-)) = 0 \textrm{ and } H^i_{\et}(-, \cI^j) = 0 \qquad \textrm{ for all } i \neq 0, j \geq 0. \]

Moreover, the two morphisms \eqref{qiAug}, \eqref{qiSuslin}, and \eqref{qiEtale} of complexes of abelian groups are all quasi-isomorphisms.

In fact, if $[X]$ is replaced by a bounded complex in $\Comp^b(\Ck)$ via the obvious use of the total complex functor, they remain quasi-isomorphisms.
\end{theo}

\begin{proof}
The assertion that $\cI^\ast$ is acyclic is \cite[Thm. 10.2]{SV}. The assertion that the $\cI^i$ are acyclic can be proven using verbatim the techniques of \cite[Prop. 3.1.7]{V1} with ``Nisnevich'' replaced by ``{\'e}tale''.

The claim that $\Gamma(k, -)$ and $\eta$ induce a quasi-isomorphisms is \cite[Thm. 7.6]{SV} (see \cite{Ge1} for a discussion on why the characteristic zero hypothesis in Suslin-Voevodsky's statement is not necessary). The morphism induced by $\epsilon$ is a quasi-isomorphism because the complex $C_\ast([X])$ is a complex of free abelian groups, i.e, a complex of projective objects.

The usual spectral sequence arguments allow us to pass from a single $[X]$ to a bounded complex.
\end{proof}

\begin{prop}\label{prop3.etreal}
Suppose that the field $k$ is algebraically closed and $n$ is invertible in $k$. For $X\in \Sch / k$ with the natural morphism $\pi:X\to \Spec(k)$, there are canonical functorial isomorphisms
\begin{align}
\rHom(\Phi^{\et}\alpha^* M(X), \nz) &= R\pi_{*}(\nz)_X, \qaq \\
\rHom(\Phi^{\et}\alpha^* M^c(X)), \nz) &= R\pi_{!}(\nz)_X. \label{McReal}
\end{align}
\end{prop}

\begin{proof}
We begin with the non-compact support case. Recall that by definition, the motive of a scheme $X \in \Sch / k$ is the image of the complex of presheaves with transfers $\underline{C}_\ast([X])$ in $\DMeffk$. The image of this complex in $\DMeffket$ is just the {\'e}tale sheafification (\cite[Cor. 5.29]{V4}) and since $k$ is separably closed, the image of this in $\Dknz$ is just its evaluation at $k$. Since the terms of this complex are free abelian groups, they are projective objects of $\Shv_{\et}(Et / k) \cong \AB$ and so $\rHom(\Phi^{\et}\alpha^* M(X), \nz)$ is represented by the complex $\hom(C_\ast([X]), \nz)$.

On the other hand, the complex $R\pi_{*}(\nz)_X$ is represented by $\cI^\ast(X)$ where $\cI^\ast$ is any complex of sheaves on $(\Sch / k)_{\et}$ for which each $\cI^i$ is an acyclic {\'e}tale sheaf and which is quasi-isomorphic to $\nz$ as a complex of {\'e}tale sheaves. For example, the $\cI^\ast$ which we chose above serves this purpose by Theorem~\ref{SVrigidity}.

So the zig-zag of quasi-isomorphisms
\[ \hom(C_\ast([X]), \nz) \stackrel{\epsilon}{\rightarrow}\  \stackrel{\Gamma(k, -)}{\leftarrow}  \ \stackrel{\eta}{\to} \cI^\ast(X)  \]
of Theorem~\ref{SVrigidity} provides the isomorphism in $\Dknz$ we are looking for.

Now let $\Sch^{comp}(k)$ be the category whose objects are closed immersions $Z \to \overline{X}$ between proper schemes and morphisms are cartesian squares. There is a canonical functor towards $\Sch^{prop}(k)$ (the category with the same objects as $\Sch(k)$ but only proper morphisms) which sends a closed immersion $Z \to \overline{X}$ to its open complement $\overline{X} - Z$. The functor $\Sch^{comp}(k) \to \Sch^{prop}(k)$ is full and essentially surjective by Nagata's compactification theorem. Furthermore, by the localisation distinguished triangle (\cite[Prop. 5.5.5]{Ke} or \cite[Prop. 4.1.5]{V1}) the functor \mbox{$M: \Sch^{comp}(k) \to \DMeffk$} (which sends a closed immersion $Z \to X$ to the image of the complex $\Tot(\underline{C}_\ast([Z]) \to \underline{C}_\ast([\overline{X}]))$) factors through \mbox{$M^c: \Sch^{prop} / k \to \DMeffk$} (via the obvious canonical natural transformation induced by the short exact sequence $0 \to [Z] \to [\overline{X}] \to  z_{equi}(X / k, 0)$). That is, $M^c(X)$ is represented by the complex $\underline{C}_\ast([X^c])$ where we define $[X^c]$ to be the two term complex $(\dots {\to} 0 {\to} [Z] {\to} [\overline{X}] {\to} 0 {\to} \dots)$.

On the {\'e}tale side, the situation is the same. The object $R\pi_{!}(\nz)_X$ is (by definition) represented by the complex $\Tot(\cI^\ast(\overline{X}) {\to} \cI^\ast(Z))$ where $\cI^\ast$ is as above.

So the isomorphism of Equation~\eqref{McReal} can be obtained from the zig-zag of quasi-isomorphisms
\[ \hom(C_\ast([X^c]), \nz) \stackrel{\epsilon}{\to}\ \stackrel{\Gamma(k, -)}{\leftarrow}  \ \stackrel{\eta}{\to} \cI^\ast(X^c).  \qedhere \]
\end{proof}

\ifShowNewSectionTwoPointTwoChanges
{ \color{blue}
\fi

\section{Truncations of twisted complexes} \label{dgcategoriesSection}

We recall here some definitions and properties of twisted complexes in dg-categories and their b{\^e}te truncations when the dg-category is negative.

\subsection{dg-categories} \label{dgcategories}

We follow the notation and conventions of \cite{Bo2} (who uses \cite{BK} as his reference).

A \emph{dg-category} is a category enriched over complexes. That is, a category for which each of the hom sets is equipped with a structure of a (cohomological)\footnote{i.e., the differential raises the degree by one.} complex of abelian groups, and such that for all objects $P, Q, R$, composition is a morphism of complexes $\hom(P, Q) \otimes \hom(Q, R) \to \hom(P, R)$. To be precise (and fix the sign convention for the differential on a tensor product of complexes), composition is bilinear, and if $f: P \to Q$ and $g: Q \to R$ are two morphisms of pure degree then
\[ \deg(gf) = \deg(f) + \deg(g), \qquad \text{ and } \qquad d(gf) = (dg)f + (-1)^{\deg(f)} g (df). \]
In particular, 
\begin{equation} \label{equa:didzero}
d(\id_E) = 0.
\end{equation}

We follow Bondarko and require dg-categories to have finite sums as well, but not all authors do this. Every additive category has a canonical structure of dg-category with all morphisms of degree zero.

A \emph{negative dg-category} is a dg-category for which there are no morphisms of pure positive degree. 

\subsection{Twisted complexes} \label{twistedcomplexes}

If $J$ is a dg-category, then a \emph{twisted complex} over $J$ is a set
\[ \{(P^i)_{i \in \ZZ}, P^i \stackrel{q_{ij}}{\to} P^j \} \]
 (where each $P^i$ is an object of $J$) such that 
\begin{enumerate}
 \item almost all the objects $P^i$ are zero,
 \item $q_{ij}$ is pure of degree $i + j - 1$,
 \item for each $i, j \in \ZZ$, one has the identity%
\footnote{An alternative formulation is to formally set $\hom(P[i], Q[j]) = \hom(P, Q)[i - j]$ (where $P[i]$ and $Q[j]$ are formal, and $\hom(P, Q)[i - j]$ means the complex $\hom(P, Q)$ shifted $i - j$ times in the usual sense and then the $q_{ij}$ form an endomorphism 
\[ q: \bigoplus_{n \in \ZZ} P^n[n] \to \bigoplus_{n \in \ZZ} P^n[n] \]
of pure degree one, and the required identity can be written as 
\[ d_J (q) + q^2 = 0 \]
where $d_J$ is the term-wise application of the differentials from the homs in the dg-category $J$ applied to each entry individually.
 } %
 \[ d_J q_{ij} + \sum_{m \in \ZZ} q_{mj}q_{im} = 0. \] 
\end{enumerate}

The notation $\{(P^i)_{i \in \ZZ}, P^i \stackrel{q_{ij}}{\to} P^j \}$ is usually shortened to $\{(P^i), q_{ij} \}$.

\subsection{Morphisms of twisted complexes} \label{morphismsoftwistedcomplexes}

A \emph{morphism of pure degree $\ell$} of twisted complexes $f: \{(P^i), q_{ij} \} \to \{(P'^i), q'_{ij} \}$ is just a morphism
\[ f = [f_{ij}]_{i, j \in \ZZ}: \bigoplus_{n \in \ZZ} P^n \to \bigoplus_{n \in \ZZ} P'^n \]
such that%
\footnote{In other words, $\deg(f_{ij}: P^i[i] \to P'^j[j]) = \ell$.} 
\[ \deg(f_{ij}: P_i \to P_j') = \ell + i - j. \]
Apart from the condition on the degrees, there are no other compatibilities the $f$ must satisfy (i.e., they don't have anything to do with $q$ or $q'$).

\subsection{Differential of morphisms of twisted complexes} \label{differentialsoftwistedcomplexes}

The differential of a morphism $f: \{(P^i), q_{ij} \} \to \{(P'^i), q'_{ij} \}$ of twisted complexes (of pure degree $\ell$) is defined by the rule%
\footnote{We used $a, b, c$ instead of $i, j, k$ to make comparison with other references easier.}%
\footnote{In matrix notation this is \[ df = d_J f + f q' - (D^\ell q D^\ell) f \] where $D$ is the diagonal matrix $D = diag(-1, 1, -1, 1, -1, \dots, -1)$.}

\[ (df)_{ab} = d_J(f_{ab}) + \sum_{c \in \ZZ} \bigl ( q'_{cb} f_{ac} - (-1)^{\ell(c-a)} f_{cb} q_{ac} \bigr ). \]

In particular, if $J$ is a \emph{negative} dg-category and $f$ is of degree zero, then we have

\[ (df)_{ab} = d_J(f_{ab}) + \sum_{a \leq c < b} q'_{cb} f_{ac} - \sum_{a < c \leq  b} f_{cb} q_{ac} \]

and if $a = b$ then this is just 
\[ (df)_{aa} = d_J(f_{aa})  \]
and if $b = a + 1$ then it is
\[ (df)_{a,a + 1} = d_J(f_{a,a + 1}) + q'_{a,a+1} f_{aa} - f_{a+1,a+1} q_{a,a+1} \]
So if $f$ is a diagonal matrix, with zero degree morphisms, then the differential is zero if and only if
\[ d f_{aa} = 0, \qquad \textrm{ and } \qquad \vcenter{\xymatrix{ P^a \ar[r]^{q_{a,a+1}} \ar[d]_{f_{aa}} & P^{a+1} \ar[d]^{f_{a+1,a+1}} \\ P'^a \ar[r]^{q_{a,a+1}} & P'^{a+1}}} \textrm{ commutes for all } a \in \ZZ. \]

A morphism $f$ is called a \emph{twisted morphism} if $\deg(f) = 0$ and $d(f) = 0$.

\subsection{Cones and shifts} \label{conesandshiftstwistedcomplexes}

If $\{(P^i), q_{ij}\}$ is a twisted complex and $n \in \ZZ$, then its \emph{shift} $\{(P^i), q_{ij}\}[n]$ is the twisted complex whose $n$th object is $P^{i+n}$ and $ij$th morphism is $(-1)^nq_{i+n, j+n}$.

Suppose that $f: \{(P^i), q_{ij} \} \to \{(P'^i), q'_{ij} \}$ is a twisted morphism of twisted complexes. The \emph{cone} of this morphism is $Cone(f) = \{(P''^i), q''_{ij} \}$ where
\[ P''^i = P^{i + 1} \oplus P'^{i}, \qquad q_{ij}'' = \left ( \begin{array}{cc} -q_{i+1,j+1} & 0 \\ f_{i+1,j} & q'_{ij} \end{array} \right ).  \]

\begin{rem}
The sign is missing in Bondarko's paper.
\end{rem}

This comes equipped with the obvious inclusion and projection morphisms
\[ \{(P'^i), q'_{ij} \} \to Cone(f) \to \{(P^i), q_{ij}\}[1] \]
which are both twisted morphisms. 

\subsection{The categories $\PT(J)$ and $\Tr(J)$.} \label{pttrJ}

Associated to any dg-category $J$ are the categories $J^0$ and $H^0(J)$ which have the same objects as $J$, but
\[ \hom_{J^0}(P, Q) = (\hom_J(P, Q))^0 \textrm{ and } \hom_{H^0(J)}(P, Q) = H^0(\hom_J(P, Q)). \] 

The category of twisted complexes with the morphisms, grading, and differential described above form again, a dg-category, which is denoted by $\PT(J)$. One then defines
\[ \Tr(J) = H^0(\PT(J)). \]
With the shift described above, and the triangles isomorphic to those of the form $P \stackrel{f}{\to} P' \to Cone(f) \to P[1]$ as distinguished triangles, the category $\Tr(J)$ is a triangulated category. If all morphisms in $J$ have degree zero, then $\Tr(J)$ is (equal to) the bounded homotopy category $K^b(J)$.

In general, there is a canonical dg-functor $J \to \PT(J)$ which induces a functor 
\begin{equation} \label{equa:heartInclusion}
 H^0J \to \Tr(J). 
 \end{equation}
Furthermore, if $J$ is a negative dg-category, then there is a canonical full, essentially surjective functor $J \to H^0(J)$ induced by the ``good'' truncation on hom complexes. In this case we obtain a functor
\begin{equation} \label{equa:weightComplexGeneral}
\Tr(J) \to K^b(H^0J)
\end{equation}
whose composition with \eqref{equa:heartInclusion} is the full inclusion of $H^0(J)$ in $K^b(H^0J)$ as complexes of degree zero.

By definition, Bondarko's weight complex functor
\begin{equation} \label{equa:weightComplex}
t: \DMeffgmkZp \to K^b(\CHMp)
\end{equation}
is the functor \eqref{equa:weightComplexGeneral} for the $J$ from Theorem~\ref{BondgEnhancement}.

\subsection{B{\^e}te truncations} \label{betetruncations}

Suppose that $J$ is a \emph{negative} differential graded category, \mbox{$f: \{(P^i), q_{i j} \} \to \{(P^{\prime i}), q'_{i j} \}$} is a morphism of twisted complexes, and $n \in \ZZ$. 

Define $P_{\geq n} = \{(P_{\geq n}^i), q_{ij}^{\geq n}\}$ and $f_{\geq n}: P_{\geq n} \to P'_{\geq n}$ (resp. $P_{\leq n} = \{(P_{\leq n}^i), q_{ij}^{\leq n}\}$ and $f_{\leq n}: P'_{\leq n} \to P'_{\leq n}$) by setting $P^i_{\geq n} = P^i, q_{ij}^{\geq n} = q_{ij}, f_{ij}^{\geq n} = f_{ij}$ (resp. $P^i_{\leq n} = P^i, q_{ij}^{\leq n} = q_{ij}, f_{ij}^{\leq n} = f_{ij}$) for $i, j \geq n$ (resp. $i, j \leq n$) and zero otherwise. Define morphisms
\[ \iota: P_{\geq n} \to P, \qquad \pi: P \to P_{\leq n}, \qquad \delta: P_{\leq {n - 1}} \to P_{\geq n}[1] \]
via the canonical inclusion \mbox{$\oplus_{i \geq n} P^i \to \oplus_{i \in \ZZ} P^i$}, projection $\oplus_{i \in \ZZ} P^i \to \oplus_{i \leq n} P^i$, and the composition $\oplus_{i \leq n - 1} P^i \to \oplus_{i \in \ZZ} P^i \stackrel{q}{\to} \oplus_{i \in \ZZ} P^i \to \oplus_{i \geq n} P^i$, respectively.

The following lemma is easily checked.

\begin{lem} We have the following elementary properties.
\begin{enumerate}
 \item $P_{\geq n}$ and $P_{\leq n}$ are again twisted complexes. That is, we have \mbox{$d_J(q_{ij}^{\geq n}) + \sum_{i < m < j} q_{mj}^{\geq n}q_{im}^{\geq n} = 0$} and $d_J( q_{ij}^{\leq n}) + \sum_{i < m < j} q_{mj}^{\leq n}q_{im}^{\leq n} = 0$.
 
 \item If $f$ is a twisted morphism, then so are $f^{\geq n}$ and $f^{\leq n}$.
 
 \item The morphisms $\iota, \pi, \delta$ are all twisted morphisms.
 
 \item We have the equality (not just an isomorphism) $Cone(\delta) = P$. Therefore, the image of 
 \begin{equation} 
P_{\geq n} \stackrel{\iota}{\to} P \stackrel{\pi}{\to} P_{\leq n - 1} \stackrel{\delta}{\to} P_{\geq n}[1] \label{equa:decompTriTruncate}
\end{equation}
is a distinguished triangle in $\Tr(J)$.

 \item We have equalities $(P_{\geq m})_{\leq m}[m] = (P_{\leq m})_{\geq m}[m] = P^m$ where $P^m$ is the object $P^m$ of $J$ considered as a twisted complex concentrated in degree zero. Hence, we also obtain the following diagrams in $\PT(J)$, giving rise to distinguished triangles in $\Tr(J)$:
\begin{align}
P_{\geq m} \stackrel{\pi'}{\to} P^m[-m] \stackrel{\delta'}{\to} P_{\geq m + 1}[1] \stackrel{\iota'}{\to} P_{\geq m}[1] \label{equa:geqTri} \\
P_{\leq m} \stackrel{\delta''}{\to} P^{m + 1}[-m] \stackrel{\iota''}{\to} P_{\leq m + 1}[1] \stackrel{\pi''}{\to} P_{\leq m}[1] \label{equa:leqTri}
\end{align}

 \item The morphisms $\iota, \pi$ induce natural transformations at the level of the dg-category of twisted complexes, i.e., for any morphism $f$ in $\PT(J)$ the diagram
 \begin{equation} \label{equa:functoriotapi}
 \xymatrix{
P_{\geq n} \ar[d]^{f^{\geq n}} \ar[r]^\iota & P  \ar[r]^\pi \ar[d]^f & P_{\leq n - 1} \ar[d]^{f^{\leq n}} 
\\
P'_{\geq n}  \ar[r]_\iota & P' \ar[r]_\pi  & P'_{\leq n - 1}  
} 
\end{equation}
commutes in $\PT(J)$.

\end{enumerate}
\end{lem}

\begin{rem}
If all morphisms in $J$ have degree zero then $\delta$ is natural for twisted morphisms, but there are very few other situations in which $\delta$ is natural.
\end{rem}

\begin{proof}
\begin{enumerate}
 \item If $i \geq n$ (resp. $j \leq n$) then this follows from the fact that $\{(P^i), q_{ij}\}$ satisfies this identity. If $j < n$ (resp. $i > n$) then all morphisms $q_{ij}^{\geq n}, q_{mj}^{\geq n}, q_{im}^{\geq n}$ (resp. $q_{ij}^{\leq n}, q_{mj}^{\leq n}, q_{im}^{\leq n}$) are zero. If $i < n \leq j$ (resp. $i \leq n < j$) then the morphisms $q_{mj}^{\geq n}$ (resp. $q_{im}^{\leq n}$) are not necessarily zero, but the compositions $q_{mj}^{\geq n}q_{im}^{\geq n}$ (resp. $q_{mj}^{\leq n}q_{im}^{\leq n}$) are, so again, we get zeros everywhere.

\item The analogous argument as for the first part works.

\item The matricies associated to $\iota$ and $\pi$ are diagonal, so since our dg-category is negative it suffices to check that the differential of each component is zero, and that they commute with the $q_{i,i + 1}$. Comutativity is clear, and identities have zero differential. For the morphism $\delta$ we do the calculation. Notice that if $a > n$ or $b < n$ then $\delta_{ab}$ and $d\delta_{ab}$ are automatically zero because their source or target is zero, so we can assume that $a \leq n \leq b$.
\[ \begin{split}
 (d\delta)_{ab} 
 &= d_J(\delta_{ab}) + \sum_{a, n \leq c < b} q_{cb} \delta_{ac} - \sum_{a < c \leq  b, n} \delta_{cb} (-1)q_{a-1,c-1} %
\\ &= d_J(f_{a-1,b}) + \sum_{a, n \leq c < b} q_{cb} q_{a-1,c} + \sum_{a < c \leq  b, n} q_{c-1,b} q_{a-1,c-1} 
\\ &= d_J(f_{a-1,b}) + \sum_{a, n \leq c < b} q_{cb} q_{a-1,c} + \sum_{a \leq c' <  b, n} q_{c',b} q_{a-1,c'} 
\\ &= d_J(f_{a-1,b}) + \sum_{a-1 < c < b} q_{cb} q_{a-1,c}
\end{split} \]
The last line is zero by virtue of $P$ being a twisted complex.

 \item This is straightforward. The $i$th object of the cone is $P^i_{\geq n} \oplus (P[1])^{i - 1}_{\leq n - 1}$ which is either $P^i_{\geq n}$ or $(P[1])^{i - 1}_{\leq n - 1}$ depending on whether $i \geq n$ or not, and in both cases we get $P^i$. For the morphisms, in the expression
 \[ f_{ij}^{Cone(\delta)} 
 = \left ( \begin{array}{cc} -(-q^{\leq n - 1}_{i+1-1,j+1-1}) & 0 \\ \delta_{i+1,j} & q^{\geq n}_{ij} \end{array} \right ) 
 = \left ( \begin{array}{cc} q^{\leq n - 1}_{i,j} & 0 \\ \delta_{i+1,j} & q^{\geq n}_{ij} \end{array} \right ) \]
the only non-zero entry is $q_{ij}^{\leq n-1} = q_{ij}$ if $i, j < n$, or is $\delta_{i + 1, j} = q_{ij}$ if $i \leq n \leq j$, or is $q_{ij}^{\geq n} = q_{ij}$ if $n \leq i, j$.

 \item The equalities $(P_{\geq m})_{\leq m}[m] = (P_{\leq m})_{\geq m}[m] = P^m$ follow immediately from the definitions. The exact triangles are just the triangle \eqref{equa:decompTriTruncate} with $n = m + 1$ applied to $P_{\geq m}$ for \eqref{equa:geqTri}, or $P_{\leq m + 1}$ for \eqref{equa:leqTri} (they are ``rotated'' as well).

 \item This is clear since $f^{\leq n}$ and $f^{\leq n-1}$ are by definition the unique morphisms fitting into the commutative diagram in $J$:
 \[ 
 \begin{gathered}[b]
 \xymatrix{
\oplus_{i \geq n} P^i \ar[r]^\iota \ar[d]_{f^{\geq n}} & \oplus_{i \in \ZZ} P^i \ar[d]^f \ar[r]^\pi \ar[d] & \oplus_{i < n} P^i \ar[d]^{f^{\leq n-1}} \\
\oplus_{i \geq n} P^i \ar[r]_\iota  & \oplus_{n \in \ZZ} P^i \ar[r]_\pi & \oplus_{i < n} P^i .
 } \\[-\dp\strutbox]
 \end{gathered}
 \qedhere
 \]
\end{enumerate}
\end{proof}

\ifShowNewSectionTwoPointTwoChanges
} 
\fi


\Addresses
\end{document}